\documentclass[fleqn,10pt]{article}
\usepackage{amsmath,amssymb,amsthm,xcolor,framed,esint,dsfont}
\usepackage[english]{babel}
\usepackage[margin=3cm]{geometry}
\usepackage{csquotes}
\usepackage{hyperref}
\usepackage{enumitem}  

\usepackage[title]{appendix}

\def\Xint#1{\mathchoice
	{\XXint\displaystyle\textstyle{#1}}%
	{\XXint\textstyle\scriptstyle{#1}}%
	{\XXint\scriptstyle\scriptscriptstyle{#1}}%
	{\XXint\scriptscriptstyle\scriptscriptstyle{#1}}%
	\!\int}

\def\XXint#1#2#3{{\setbox0=\hbox{$#1{#2#3}{\int}$}
		\vcenter{\hbox{$#2#3$}}\kern-.5\wd0}}

\newcommand{\norm}[1]{{\Vert #1\Vert}}
\newcommand{\abs}[1]{{\left\vert #1\right\vert}}
\newcommand{\R}{{\mathbb R}}

\newcommand{\ep}{\varepsilon}
\newcommand{\lt}{\left}
\newcommand{\rt}{\right}
\newcommand{\na}{\nabla}
\newcommand{\nn}{\nonumber}
\newcommand{\one}{{\mathds{1}}}
\newcommand{\e}{\varepsilon}
\DeclareMathOperator{\dv}{div}
\DeclareMathOperator{\supp}{supp}

\DeclareMathOperator{\dist}{dist}

\newcommand{\loc}{{\rm loc}}
\newcommand{\chara}{1\!\!1}

\newcommand{\nl}{\newline}

\newcommand{\lm}{\lambda}

\newcommand{\qd}{\quad}

\newcommand{\wt}{\widetilde}

\newcommand{\PPI}{\mathcal{P}}

\newcommand{\OI}{\mathcal{O}}

\newcommand{\B}{\mathcal{B}}

\newcommand{\ca}[1]{\mathrm{Card}\lt(#1\rt)}

\newcommand{\ti}{\tilde}

\newtheorem{thm}{Theorem}
\newtheorem{prop}[thm]{Proposition}
\newtheorem{lem}[thm]{Lemma}

\theoremstyle{definition}

\newtheorem{rem}[thm]{Remark}

\title{Rigidity of a non-elliptic differential inclusion related to the Aviles-Giga conjecture}
\date{}
\author{Xavier Lamy\footnote{Institut de Math\'ematiques de Toulouse, UMR 5219, Universit\'e de Toulouse, CNRS, UPS
IMT, F-31062 Toulouse Cedex 9, France. Email: Xavier.Lamy@math.univ-toulouse.fr}
\and Andrew Lorent\footnote{Department of Mathematical Sciences, University of Cincinnati, Cincinnati, OH 45221, USA. Email: lorentaw@uc.edu} 
\and Guanying Peng\footnote{Department of Mathematics, University of Arizona, Tucson, AZ 85721, USA. Email: gypeng@math.arizona.edu}}

\begin{document}

\maketitle
\begin{abstract}
In this paper we prove sharp regularity for a differential inclusion into a set $K\subset\R^{2\times 2}$ that arises in connection with the Aviles-Giga functional.
The set $K$ is not elliptic, and in that sense our main result goes beyond \v{S}ver\'{a}k's regularity theorem on elliptic differential inclusions. It can also be reformulated as a sharp regularity result for a critical nonlinear Beltrami equation.
In terms of the Aviles-Giga energy, our main result implies that zero energy states  coincide (modulo a canonical transformation) with solutions of the differential inclusion into $K$. This opens new perspectives towards understanding energy concentration properties for Aviles-Giga: quantitative estimates for the stability of zero energy states can now be approached from the point of view of stability estimates for differential inclusions.
All these reformulations of our results are strong improvements of a recent work by the last two authors Lorent and Peng, where the link between the differential inclusion into $K$ and the Aviles-Giga functional was first observed and used.
Our proof relies moreover on new observations concerning the algebraic structure of entropies.
\end{abstract}

\section{Introduction} The Aviles-Giga functional for $u\in W^{2,2}(\Omega)$ over a bounded domain $\Omega\subset\R^2$ is given by
\begin{equation*}
I_{\ep}(u)=\int_{\Omega}\lt(\ep \lt|\na^2 u\rt|^2  +\frac{\lt(1-\lt|\na u\rt|^2\rt)^2}{\ep}\rt)\; dx.
\end{equation*}
Here $\nabla^2 u$ is the Hessian matrix of the scalar-valued function $u$ and $\ep>0$ is a small parameter. This is a second order functional that (subject to appropriate boundary conditions) models phenomena from thin film blistering to smectic liquid crystals, and is also the most natural higher order generalization of the Cahn-Hilliard functional. The Aviles-Giga conjecture for the $\Gamma$-limit of $I_{\ep}$ is one of the central conjectures in the theory of $\Gamma$-convergence and has attracted a great deal of attention; see for example  \cite{avgig0,avgig1,ADM,mul2,ottodel1}. One of the main theorems in the theory of the Aviles-Giga functional is the characterization of 
\enquote{zero energy states} of the functional by Jabin,  Otto and Perthame \cite{otto}. A zero energy state is a function $u$ that is a strong limit of a sequence $u_\ep$ with $I_{\ep}(u_\ep)\rightarrow 0$ as $\ep\rightarrow 0$. Clearly $u$ satisfies the Eikonal equation given by
\begin{equation}\label{eq10}
\lt|\na u\rt|=1 \qd\text{a.e.}
\end{equation}
A formulation of the Jabin-Otto-Perthame theorem involves the notion of \emph{entropies}, which is a central tool for the analysis of the Aviles-Giga functional. Non-technically speaking, entropies are smooth vector fields $\Phi:\R^2\rightarrow\R^2$ such that $\dv\Phi(\nabla^{\perp} u)=0$ if $u$ is a smooth solution to the Eikonal equation. 
For  weak solutions $u=\lim_{\ep\rightarrow 0} u_\e$ with $\sup_{\ep}I_{\ep}(u_\ep)<\infty$, $\dv\Phi(\nabla^{\perp} u)$ are measures, called \emph{entropy measures}, that detect the jump in $\nabla u$  (see \eqref{eq14} and \eqref{eq15} for a detailed definition of entropies). The Jabin-Otto-Perthame theorem states that
if $u$ is a solution to the Eikonal equation and if for every entropy  $\Phi$ the function $u$ satisfies $\dv\Phi(\na^{\perp} u)=0$  
distributionally in $\Omega$, then $\na u$ is  smooth outside a locally finite set. Indeed in any convex neighborhood $U$ of a singular point $x_0$  the vector field $\na^{\perp} u$ forms a vortex around $x_0$ in $U$. 

Recently the second two authors provided a generalization of this result in \cite{LP}: the same conclusion holds under the weaker assumption that  $m=\nabla^\perp u$ satisfies
\begin{equation}
\label{eqintroa1}
\dv\Sigma_j(m)=0\text{ distributionally in }\Omega\text{ for }j=1,2,
\end{equation}
where $\Sigma_1, \Sigma_2 \in C^{\infty}(\R^2;\R^2)$ are the entropies introduced by Jin and Kohn \cite{kohn} (see \eqref{eqx10} in Section \ref{background} below) and further used by Ambrosio, De Lellis  and Mantegazza \cite{ADM} to formulate a $\Gamma$-limit conjecture for the Aviles-Giga functional. 
A necessary condition for their conjecture to hold 
 is that the Jin-Kohn entropy productions $\dv\Sigma_j(\na^{\perp} u)$, if they are measures, control all other entropy productions $\dv\Phi(\na^{\perp} u)$. Hence the main result in \cite{LP} shows this in the particular case when the Jin-Kohn entropy productions are zero.
 A key new perspective in \cite{LP} is to associate to a function $u$ satisfying $\dv\Sigma_j(\na^{\perp} u)=0$  and $\abs{\nabla u}=1$ a mapping $F:\Omega\rightarrow\R^2$ that satisfies a differential inclusion into a set $K$ (see \eqref{eq2}) determined by the two Jin-Kohn entropies $\Sigma_j$. 
  The main result in \cite{LP} shows regularity for any $F$ satisfying the differential inclusion $DF\in K$, provided $F$ was originally associated to a function $u$ as above. This is the case if $F$ already has some regularity, e.g. $F\in W^{2,1}$ \cite[Theorem~5]{LP}.
Our aim in the present work is to prove a more natural regularity result for the differential inclusion $DF\in K$ (regardless of whether $F$ was originally associated to a function $u$), removing this extra regularity assumption.

Note that the Eikonal equation \eqref{eq10} can be equivalently formulated as 
\begin{equation}\label{eq11}
|m|=1\text{ a.e.}, \qd\dv m=0
\end{equation}
by identifying $m=\nabla^{\perp} u$. 
 In this setting the main result of \cite{LP} shows regularity of $m$ satisfying \eqref{eq11} and \eqref{eqintroa1}. There is a  correspondence between Lipschitz maps $F$ satisfying the above-mentioned differential inclusion $DF\in K$ a.e., and unit vector fields $m$ satisfying \eqref{eqintroa1} but \emph{not necessarily divergence free}. Hence proving a natural regularity result for the differential inclusion into $K$ amounts to generalizing the main result in \cite{LP} by removing the assumption that $\dv m=0$. This is one of the formulations of our main results: if a vector field $m\colon\Omega\to\mathbb S^1$ satisfies \eqref{eqintroa1}, then once again the regularity and rigidity of zero energy states are valid (see Theorem~\ref{TT0}).

Formulated in terms of differential inclusions (see Theorem~\ref{TT1}), this constitutes a sharp regularity result for the differential inclusion into $K$, compared to the corresponding one in \cite{LP}. The set $K$ is \em not  elliptic \rm in the sense of \v{S}ver\'{a}k \cite{sverak} and DiPerna \cite{dp}. As such our Theorem \ref{TT1} is (to our knowledge) the first regularity/rigidity result for non-elliptic differential inclusions and opens the possibility of regularity results for differential inclusions under much more general hypotheses than those of \cite{sverak}.  

However our principal aim is the study of the Aviles-Giga conjecture. As will be explained, we envision our main result as a technical tool in the study of the   energy concentration. Specifically we are interested to attack this problem by establishing quantitative stability estimates for the differential inclusion into $K$. The first step in such a program is to establish rigidity of the differential inclusion into $K$ itself,  and is the purpose of the present work.

\subsection{Statement of the main results}
\label{background}

To state our main result, let us first introduce the Jin-Kohn entropies $\Sigma_1,\Sigma_2\in C^{\infty}(\R^2;\R^2)$. For $v\in\R^2$, the two entropies are given by
\begin{equation}
\label{eqx10}
\begin{aligned}
\Sigma_1(v)&=\left( v_2\lt(1-v_1^2-\frac{v_2^2}{3}\rt) ,  v_1\lt(1-v_2^2-\frac{v_1^2}{3}\rt) \right),\\
\Sigma_2(v)&=\lt( -v_1\lt(1-\frac{2 v_1^2}{3}\rt),  v_2\lt(1-\frac{2 v_2^2}{3}\rt)  \rt).
\end{aligned}
\end{equation}
 Note that $\Sigma_2(v) = R_{\frac\pi 4}\Sigma_1( R_{-\frac\pi 4}v)$, where $R_{\theta}$ denotes the rotation of angle $\theta$.
As mentioned earlier, our main result can be stated either in terms of unit vector fields $m$ that are \emph{not  necessarily divergence free}, or in terms of differential inclusions (see also Theorem~\ref{TT2} below in terms of nonlinear Beltrami equations). We first adopt the unit vector field point of view:
\begin{thm}\label{TT0}
Let $\Omega\subset\R^2$ be an open set and $m\colon \Omega\to\mathbb R^2$ satisfy
\begin{equation}
\label{eq12}
\abs{m}=1\quad\text{a.e.},\qquad \dv\Sigma_j(m)=0\quad\text{in }\mathcal D'(\Omega)\quad\text{for }j=1,2.
\end{equation}
Then $m$ is locally  Lipschitz outside a locally finite set of points $S$. 
Moreover, for any singular point $\zeta\in S$ there exists $\alpha\in \lbrace \pm 1\rbrace$ such that in  any convex neighborhood $ \OI\subset\subset\Omega$ of $\zeta$,
\begin{equation*}
m(x)=\alpha i\frac{x-\zeta}{\abs{x-\zeta}}\quad\text{for all }x\in\OI,
\end{equation*}
where $i\in \mathbb C$ is identified with the  counterclockwise rotation of angle $\frac\pi 2$ in $\R^2$.
\end{thm}

 As mentioned earlier, this result is a strong extension of the main result in \cite{LP}, which states the same regularity for $m$ satisfying the additional divergence free assumption, i.e., $m$ satisfying the Eikonal equation, and it opens new perspectives towards energy concentration results for Aviles-Giga minimizers.

To give the equivalent statement in terms of differential inclusions, we introduce the mapping $P: [0, 2\pi)\rightarrow \R^{2\times 2}$ given by
\begin{align}
P(\theta)&:=\left( \begin{array}{c}
i\Sigma_1(e^{i\theta}) \\
i\Sigma_2(e^{i\theta})
\end{array}\right)\nn\\
&=\lt(\begin{array}{cc} -\frac{2}{3}\cos^3(\theta)  & \frac{2}{3}\sin^3(\theta)  \\ -\sin(\theta)\lt(1-\frac{2}{3}\sin^2(\theta)\rt)  & -\cos(\theta)\lt(1-\frac{2}{3}\cos^2(\theta)\rt)\end{array}\rt) \label{eq1},
\end{align}
and we define
\begin{equation}
\label{eq2}
K:=\lt\{P(\theta):\theta\in \lt[0,2\pi\rt)\rt\}\subset\R^{2\times 2}.
\end{equation}

\begin{thm}
\label{TT1}
Let $\Omega \subset \R^2$ be an open set, and
 $F\colon\Omega\to\R^2$ be a Lipschitz map  satisfying the differential inclusion $DF\in K$ a.e. Then $DF$ is locally  Lipschitz outside a locally finite set of points $S$. 
Moreover,  for any singular point $\zeta\in S$ there exists $\alpha\in \lbrace \pm 1\rbrace$ such that in  any convex neighborhood $ \OI\subset\subset\Omega$ of $\zeta$, 
\begin{equation}
\label{eq210}
\lt(\begin{array}{c} F_{1,1}(x)+F_{2,2}(x) \\ F_{2,1}(x)-F_{1,2}(x) \end{array}\rt)=\alpha i \frac{x-\zeta}{\lt|x-\zeta\rt|}\quad\text{for all }x\in \OI. 
\end{equation}
\end{thm}

As already explained, 
Theorem \ref{TT1} is a considerable improvement on the corresponding  result in \cite{LP}, where the same rigidity was proved under the additional -- and unnatural -- assumption $F\in W^{2,1}(\Omega)$. Our Theorem \ref{TT1}, in contrast, is a sharp regularity result. In the following, we also reformulate this result as sharp regularity for a nonlinear Beltrami equation;  see Theorem~\ref{TT2} below.

\begin{rem}\label{rmk3}
Since Theorems~\ref{TT0} and \ref{TT1} are local statements, in the sequel we will assume without loss of generality that $\Omega$ is smooth, bounded and simply connected.
\end{rem}

\subsection{Differential inclusions} 

Regularity of differential inclusions is a classical subject. Let
\begin{equation*}
CO_{+}(n):=\lt\{\lm R: R\in SO(n)\rt\}.
\end{equation*} 
The first and best known result is analyticity of the differential 
inclusion $Du\in CO_{+}(2)$. This differential inclusion is nothing other than the Cauchy-Riemann equations and analyticity is one of the first basic theorems of complex analysis. Rigidity of the differential inclusion $Du\in CO_{+}(3)$ was studied by Liouville in 1850 for $C^3$ mappings \cite{liou}. The generalization of this result has been a topic of great interest in the Quasiconformal analysis community \cite{gehring,resh,boj,iw1,msy}. Another  classical example is  the differential inclusion into the set $\lt\{A\in \R^{n\times n}: A^T=A, \mathrm{Tr}(A)=0\rt\}$,  which corresponds to the Laplace equation $\Delta u=0$ in $\R^n$. 

We say a set $S\subset \R^{m\times n}$ has a 
\emph{Rank-$1$ connection} if there exist $A,B\in S$ with $\mathrm{Rank}(A-B)=1$. It is a simple fact that if a set $S$ has a Rank-$1$ connection then one can construct wild solutions to the differential inclusion $Dw\in S$ through the construction of laminates. 
It is tempting to conjecture that if a set $S$ 
has no Rank-$1$ connections then the differential inclusion $Dw\in S$ has higher regularity, but this is completely false.  This falsity is a key fact in the recent spectacular progress in counter-examples to regularity of PDE \cite{mulsv2, sz, mulsv3}. In $\R^{2\times 2}$, $\mathrm{Rank}(A-B)=1$ if and only if $\det(A-B)=0$. So for any connected analytic set $S\subset \R^{2\times 2}$ without Rank-$1$ connections, by Lojasiewicz inequality (up to the change of sign) there exists some $p\in \mathbb{N}$ such that $\det(A-B)\gtrsim \lt|A-B\rt|^p$. For an elliptic set, which is a set whose tangent space at any point does not contain Rank-$1$ connections, the inequality holds for $p=2$. In \cite{sverak} \v{S}ver\'{a}k proved 
\begin{thm}[\v{S}ver\'{a}k \cite{sverak}]
	\label{svdiffinc}
	Let $\Omega\subset\R^2$ be open and bounded and $S\subset\R^{2\times 2}$ be a closed connected smooth submanifold without Rank-$1$ connections. Further assume that $S$ is elliptic. Then every Lipschitz $w:\Omega\rightarrow \R^2$ satisfying $Dw\in S$ a.e.  is smooth.
\end{thm}
 Since $CO_{+}(2)$ is a closed smooth connected elliptic set, this result is a far reaching generalization of analyticity of Lipschitz solutions of the Cauchy-Riemann equations (for more details on this topic see \cite{lor10}). To our knowledge this is the most general result on regularity of differential inclusions. The set $K$ defined by \eqref{eq2} is not elliptic and easy examples show that the regularity provided by 
Theorem \ref{TT1} is optimal: differential inclusions into $K$ do have singularities, however the nature of the singularities is explicitly given by \eqref{eq210} on any given convex subdomain $\OI\subset\subset \Omega$ containing a singularity $\zeta$. Another way to formulate this is as follows. Let $\theta(x)$ be defined by $P(\theta(x))=DF(x)$, then for some $\alpha\in \lt\{1,-1\rt\}$, 
\begin{equation*}
e^{i\theta(x)}=\alpha i \frac{x-\zeta}{\lt|x-\zeta\rt|}\qd\text{ for all }x\in \OI,
\end{equation*}
i.e.\ $e^{i\theta(x)}$ forms a vortex around $\zeta$. 

Theorem \ref{TT1} can also be formulated as a result for a nonlinear Beltrami equation as done in \cite{LP}. Namely, we also have
\begin{thm}
	\label{TT2}
	Given an open set $\ti\Omega\subset \mathbb{C}$,  let $v:\ti\Omega\to\mathbb{C}$ be a Lipschitz function that satisfies the nonlinear Beltrami system 
	\begin{equation}
	\label{eqy60}
	\frac{\partial v}{\partial \bar{z}}(z)=\frac{4}{3}\lt(\frac{\partial v}{\partial z}(z) \rt)^3,\qd\lt|\frac{\partial v}{\partial z}(z)\rt|=\frac{1}{2}\qd\text{ for a.e. }z\in \ti\Omega, 
	\end{equation}
	 then $\na v$ is locally Lipschitz outside a discrete set $S$,
	  and for any convex set $\OI\subset\subset \ti\Omega$, we have $\ca{\OI\cap S}=1$. 
\end{thm}
This result is again a strong improvement of the corresponding result in \cite{LP}, where the same regularity for $v$ was established under an additional $W^{2,1}$ regularity assumption on $v$. 
Equation \eqref{eqy60} can be recognized as a nonlinear Beltrami system by introducing $\mathcal{H}_0\lt(\xi\rt):=\frac{4}{3}\xi^3$. Then this equation can be written as $\frac{\partial v}{\partial \bar{z}}(z)=\mathcal{H}_0\lt(\frac{\partial v}{\partial z}(z)\rt)$, $\lt|\frac{\partial v}{\partial z} \rt|=\frac{1}{2}$ a.e. Equations of the form  $\frac{\partial v}{\partial \bar{z}}=\mathcal{H}\lt(z,\frac{\partial v}{\partial z}\rt)$ have received a great deal of study in the last few years.  Under the assumptions that 
\begin{enumerate}[label=(\roman*)]
	\item  $z\mapsto \mathcal{H}\lt(z,w\rt)$ is measurable,
	\item  and for $w_1,w_2\in \mathbb{C}$, $\lt|\mathcal{H}\lt(z,w_1\rt)-\mathcal{H}\lt(z,w_2\rt)\rt|\leq k\lt|w_1-w_2\rt|$ for some $k<1$,
\end{enumerate}
a powerful existence and regularity theory of nonlinear Beltrami equations has been developed; see \cite{boj1,iwan,boj2,ast,ast2}. 
Our system \eqref{eqy60} corresponds to a critical case,
 since the mapping $\mathcal{H}_0$ has Lipschitz constant $1$ on the circle with radius $\frac{1}{2}$. 
As such Theorem \ref{TT2} does not follow from any of the known regularity results,  and is to our knowledge the first to hold in the critical  and genuinely nonlinear case. For linear Beltrami equations there are many powerful results in the critical case, see \cite[Chapter~20]{ast2}. Note also that the power 3 nonlinearity $\mathcal H_0$ appears as an interesting particular case in \cite[Remark~15]{ACFJK}.

\subsection{The Aviles-Giga functional} 
As noted at the beginning the Aviles-Giga functional is a higher order generalization of the Cahn-Hilliard functional. In 1977 Modica-Mortola \cite{modmor} proved that the Cahn-Hilliard functional $\Gamma$-converges to the surface area of the jump set of the limiting function. This proved a conjecture of De Giorgi \cite{degio} and was one of the first results in $\Gamma$-convergence. Since then vast literature in applying $\Gamma$-convergence to problems in Calculus of Variations and PDE has evolved. One of the main conjectures in the field of $\Gamma$-convergence is, 
loosely stated, the conjecture that the $\Gamma$-limit of the Aviles-Giga functional is an energy functional  of the form
\begin{equation*}
I_0(\nabla u)=c\int_{J_{\nabla u}}\abs{\nabla u^+ -\nabla u^-}^3 d\mathcal H^1\qquad\text{for }u\text{ solving }\abs{\nabla u}=1,
\end{equation*}
where $J_{\nabla u}$ is a one-dimensional jump set, and $\nabla u^\pm$ denote traces of $\nabla u$ on each side of the jump set.
 The principal reason that makes the Aviles-Giga conjecture much more difficult than the $\Gamma$-convergence of the Cahn-Hilliard functional is that the 
 power $3$ scaling of the Aviles-Giga functional makes the BV function theory inapplicable, and it is not even clear that $\na u$ has a one-dimensional jump set. If one assumes that $\nabla u$ is $BV$, then the conjecture is settled in \cite{ADM, ark,contidel}. However a strong limit $u$ of a sequence of bounded Aviles-Giga energy does not in general satisfy $\nabla u\in BV$ (see \cite{ADM}), and despite considerable efforts from multiple authors  \cite{avgig0,avgig1,ADM,mul2,ottodel1} the $\Gamma$-convergence conjecture remains very much open. 
 
 Similar questions and open problems arise in the context of a micromagnetics energy first studied by Rivi\`ere and Serfaty \cite{ser1,ser3}. There some issues are simplified due to the fact that vortices can not appear in the limit, but the works on both problems have certainly influenced each other (see e.g. the rectifiability results \cite{ottodel1,AKLR02}). Analogous issues are also of importance in the study of large deviation principles for some stochastic processes, where the limiting equations are scalar conservation laws \cite{BBMN10}.

The precise conjecture in \cite{ADM} is that the $\Gamma$-limit is (up to a constant) the total mass of the entropy measure
\begin{equation*}
\mu = \left\Vert \begin{array}{c} \dv\Sigma_1(\nabla^\perp u)\\\dv\Sigma_2(\nabla^\perp u)\end{array}\right\Vert,
\end{equation*}
which is indeed controlled by the energy, and coincides with the cubic jump cost when $\nabla u\in BV$.
The main choke point for progress on the conjecture is the lack of methods or tools to show that
 $\mu$ is concentrated on a rectifiable one-dimensional set. In this direction, De Lellis and Otto prove in \cite{ottodel1} that the points of positive one-dimensional density for $\mu$ do form an $\mathcal H^1$-rectifiable set, but so far concentration for $\mu$ remains completely out of reach. It is not even known that $\mu$ is singular with respect to the Lebesgue measure. Progress towards such concentration results is in truth our main motivation.

Analogous questions can be studied for weak solutions of Burgers' equation, motivated by similar $\Gamma$-convergence conjectures related to large deviation principles for some stochastic processes; see \cite{BBMN10} and the references therein. There one-dimensional concentration of the entropy measure is also open but it is shown in \cite{xavotto} that the set of non-Lebesgue points has dimension at most one (very recently this has been extended to more general conservation laws in \cite{marconi19}). In the Aviles-Giga setting such result is not known yet.

The most natural way to tackle the problem of concentration is to prove a  Poincar\'e type inequality,  that would bound the distance of $u$ to zero energy states in terms of the Aviles-Giga energy.  
This is in the spirit of what was achieved in \cite{xavotto} in the context of Burgers' equation, and this was also the motivation for \cite{lor25}.  
Part of our interest in proving Theorem \ref{TT1} is to develop a new tool to establish such an inequality  in the Aviles-Giga setting. 
Specifically we are motivated by the recent powerful  quantitative  stability estimates for the rigidity of differential inclusion into $SO(n)$, $CO_{+}(n)$ \cite{MFJ,farazh} and have a view to proving a  stability estimate for the differential inclusion into $K$. 

The first step in proving quantitative  stability is to show rigidity for the differential inclusion itself and that is what is achieved in Theorem \ref{TT1}.  Our hope is to obtain in a future step a stability estimate of the form
\begin{equation}\label{eq:hopestability}
\inf_{\lbrace G\in W^{1,\infty}(B_1)\,\text{s.t.}\,\nabla G\in K\rbrace} \int_{B_1}\abs{\nabla F -\nabla G}\lesssim \left(\int_{B_1}\dist(\nabla F,K)\right)^\alpha
\end{equation}
for all Lipschitz $F$ and for some $\alpha>0$.
Then the crucial interest of Theorem~\ref{TT0}/\ref{TT1} is that it tells us that the states of exact differential inclusion coincide (modulo a canonical transformation) with the zero energy states for Aviles-Giga. Therefore combining \eqref{eq:hopestability} with 
the quantitative Hodge estimate in \cite[Theorem~4.3]{ADM}  would imply 
\begin{equation*}
 \inf_{\lbrace \nabla v \text{ zero energy state}\rbrace}\int_{B_1}\abs{\nabla u-\nabla v}\lesssim \mu(B_1)^\alpha,
\end{equation*}
which is the above mentioned Poincar\'e type inequality.

\subsection{Proof sketch and plan of the article}

Throughout this paper, we use the notation $A\lesssim B$ to indicate $A\leq cB$ for some constant $c$ independent of the underlying domain or functions. Recall that our goal is to show regularity and rigidity of a vector field $m$ that satisfies
\begin{equation}\label{sketch:equation}
\abs{m}=1\text{ a.e.},\quad\text{and }\qd\dv\Sigma_j(m)=0\text{ for }j=1,2.
\end{equation}
In \cite{LP} this was achieved under the additional assumption that $\dv m=0$. Thus the proof of Theorem~\ref{TT0} will consist in proving that $\dv m=0$ so that we can appeal to \cite{LP} to conclude. An indication that this should be true is the explicit identity
\begin{align}\label{sketch:identity}
\dv m & =-2m_1m_2 \dv\Sigma_1(m) \nn\\
&\quad  + (m_1^2-m_2^2)\dv\Sigma_2(m)\qquad\text{if }m\colon\R^2\to\mathbb S^1\text{ is smooth.}
\end{align}
If $m$ is not regular enough to apply the chain rule, this identity can not be computed. It is then natural to try approximating $m$ with a mollification $m_\e=m*\rho_\e$. But $m_\e$ does not take values into $\mathbb S^1$, and a lot of additional terms appear in the identity \eqref{sketch:identity}. These terms involve the lack of \enquote{$\mathbb S^1$-valuedness} through the nonlinear commutator
\begin{equation*}
1-\abs{m_\e}^2 = [\Pi(m)]_\e -\Pi(m_\e),\qquad \Pi=\abs{\cdot}^2.
\end{equation*}
It was remarked in \cite{DeI} that if $m$ is \enquote{$\frac 13$-differentiable} in a strong enough sense ($W^{\frac 13,3}$ in that case), then commutator estimates (introduced in \cite{titi} in the context of Euler's equation) imply that such additional terms vanish in the limit $\e\to 0$. Here some regularity of $m$ is available thanks to \cite{LP}, where it is shown that any $m$ satisfying \eqref{sketch:equation} has the Besov regularity $B^{\frac 13}_{4,\infty}$. (This is related to a weak coercivity property of the differential inclusion into $K$, namely $\det(A-B)\gtrsim \abs{A-B}^4$ for all $A,B\in K$, and to standard compensated compactness tools for estimating determinants. See \cite{FK} for a regularity result in a similar spirit.) This Besov regularity  does not imply the $W^{\frac 13,3}$ regularity used in \cite{DeI}; it is not good enough to ensure that the commutator terms tend to zero, and to obtain \eqref{sketch:identity} for our map $m$. It is however good enough to bound the commutator terms in order to deduce that
\begin{equation*}
\dv m\in L^{\frac 43}_{\loc},
\end{equation*}
and this constitutes the first step of our proof in Section~\ref{s:controldiv}. In the same spirit, throughout the article we make extensive use of commutator estimates to derive useful information from identities that are valid for smooth $\mathbb S^1$-valued maps.

The rest of the proof is to obtain $\dv m=0$ from this preliminary estimate. To that end we use a tool already crucial in \cite{mul2,otto,ignatmerlet12,DeI,LP,GL}, namely entropies and entropy productions.
The terminology comes from an analogy with scalar conservation laws, where similar objects play an important role.
 An entropy is a $C^2$ map $\Phi\colon\mathbb S^1\to\R^2$ that provides an admissible renormalization  of the Eikonal equation \eqref{eq11}
in the sense (similar to \cite{dipernalions} for transport equations) that $\dv\Phi(m)=0$ for all smooth solutions of the Eikonal equation.  Applying the chain rule one sees that this is equivalent to the requirement
\begin{equation}\label{eq14}
e^{it}\cdot\frac{d}{dt}\Phi(e^{it})=0.
\end{equation} 
If a solution $m$ is $BV$, one can still apply the chain rule and see that $\dv\Phi(m)$ is concentrated on the jump set $J_m$. For instance the Jin-Kohn entropies \eqref{eqx10} are entropies in that sense. Note that here we will be computing entropy productions $\dv\Phi(m)$ of unit vector fields $m$ that may not be divergence free, so additional terms involving $\dv m$  will appear. But since $\dv m\in L^{\frac 43}_{\loc}$ we can use commutator estimates as described above in the spirit of \cite{DeI} to deduce that $\dv\Phi(m)\in L^{\frac 43}_{\loc}$ for our map $m$ and any entropy $\Phi$.  Refining the commutator estimates from \cite{DeI}, we obtain in fact a more precise pointwise bound.
 Specifically, for all $C^2$ entropies $\Phi$, we have
\begin{equation}\label{sketch:controlent}
\abs{\dv\Phi(m)}\lesssim \norm{\Phi}_{C^2(\mathbb S^1)}\mathcal P\quad\text{a.e., for some }\mathcal P\in L^{\frac 43}_{\loc}.
\end{equation}
This is achieved in Section~\ref{s:controlent}.

Note that to obtain \eqref{sketch:controlent} we mollify $m$
and since the mollified $m_\e$ take values into $\overline B_1$ instead of $\mathbb S^1$, one first needs to extend $\Phi$ to $\overline B_1$. For \eqref{sketch:controlent} the choice of an extension is quite flexible, but the rest of our proof relies crucially on choosing extensions with specific algebraic properties to obtain more information.  A key observation in \cite{LP} is that a special family of extended entropies $\wt\Phi\colon \R^2\to\R^2$, called \emph{harmonic} entropies, satisfy the identity
\begin{align}\label{sketch:identityharm}
\dv\wt\Phi(m)=A^{\wt\Phi}(m)\dv m &+ F_1^{\wt\Phi}(m)\dv\Sigma_1(m)\nn\\
&+F_2^{\wt\Phi}(m)\dv\Sigma_2(m)\qquad\text{if }m\colon\R^2\to\mathbb S^1 \text{ is smooth},
\end{align}
where $A^{\wt\Phi}$, $F_1^{\wt\Phi}$ and $F_2^{\wt\Phi}$ are smooth functions depending on the harmonic entropy $\wt \Phi$. 
Using \eqref{sketch:identityharm} and commutator estimates and recalling that $\dv\Sigma_j(m)=0$, we are then able to compute 
\begin{equation}\label{sketch:compent}
\dv\wt\Phi(m)=A^{\wt\Phi}(m) \dv m\qquad \text{a.e.},
\end{equation}
for our map $m$.  We prove this identity in Section~\ref{s:compent}.

The conclusion of the proof follows in Section~\ref{s:proof}, where we remark that for any fixed $x$, the explicit linear map  $\wt\Phi\mapsto A^{\wt\Phi}(m(x))$ can not satisfy a bound of the form $\abs{A^{\wt\Phi}(m(x))}\lesssim\norm{\wt\Phi_{\lfloor\mathbb S^1}}_{C^2(\mathbb S^1)}$. This is related to the fact that the Hilbert transform (or conjugate function operator) on $\mathbb S^1$ is not bounded from $C^0(\mathbb S^1)$ to $L^\infty(\mathbb S^1)$ (see e.g. \cite[\S~VII]{zygmund}). As a consequence, the only possibility for \eqref{sketch:compent} and \eqref{sketch:controlent} to be compatible is that $\dv m=0$ a.e., and we are then in a situation to apply the rigidity result in \cite{LP}.\nl

\emph{Acknowledgements.} 
 We warmly thank the anonymous referee for suggesting significant simplifications,  in particular pointing out that one of our key arguments could be directly formulated in terms of the Hilbert transform without resorting to superfluous intermediate steps. 
A.L. gratefully acknowledges the support  of the Simons foundation, collaboration grant \#426900. X.L. was supported in part by ANR project ANR-18-CE40-0023 and COOPINTER project IEA-297303.

\section{Control of $\dv m$ in $L^{\frac 43}$}\label{s:controldiv}

In this section we obtain some preliminary $L^{\frac 43}$ control of $\dv m$ for $m$ satisfying the assumptions of Theorem \ref{TT0}. Let $\rho$ be the standard convolution kernel, and let $\rho_{\ep}(z)=\rho\lt(\frac{z}{\ep}\rt)\ep^{-2}$. 
Given a function $f$ we let $[f]_{\ep}:=f*\rho_{\ep}$.  Let us first recall the definition of Besov spaces on domains \cite{triebel06}. Given a bounded Lipschitz domain $\Omega\subset\R^n$ and $f:\Omega\rightarrow \R$, $z\in\R^n$, we define 
\begin{equation*}
D^zf(x):=\begin{cases}
f(x+z)-f(x) &\text{ if } x,x+z\in\Omega;\\
0 &\text{ otherwise}.
\end{cases}
\end{equation*}
For any $s\in (0,1)$ $p,q\in [1,\infty]$, we set
\begin{equation*}
\abs{f}_{B^s_{p,q}(\Omega)}=\left\Vert t^{-s}\sup_{\abs{h}\leq t}\norm{D^h f}_{L^p(\Omega)} \right\Vert_{L^q(\frac{dt}{t})},
\end{equation*}
and  the Besov space $B^s_{p,q}(\Omega)$ is the space of functions $f\colon\Omega\to\R$ such that
\begin{equation*}
\norm{f}_{L^p(\Omega)} + \abs{f}_{B^{s}_{p,q}(\Omega)} <\infty.
\end{equation*}
In the sequel we will mostly use the space $B^{\frac 13}_{4,\infty}$. The main result of this section is

\begin{prop}\label{p:controldiv}
	Let $m:\Omega\rightarrow\R^2$ satisfy \eqref{eq12}, then
	\begin{equation*}
	m\in B^{\frac 13}_{4,\infty,\loc}(\Omega)\quad\text{and}\quad\dv m\in L^{\frac{4}{3}}_{\loc}(\Omega). 
	\end{equation*}
	Moreover we have
	\begin{equation}\label{eq340}
	\abs{\dv m}\lesssim \PPI\qquad\text{a.e. in }\Omega,
	\end{equation}
	where $\PPI\in L^{\frac 43}_{\loc}(\Omega)$ is any weak $L^{\frac 43}_{\loc}$ limit of a subsequence of
	\begin{equation}\label{eq600}
	\PPI_\e (x):=\ep^{-1}\Xint{-}_{B_{\ep}(0)} \lt|D^z m(x)\rt|^3 dz,
	\end{equation}
	as $\e\to 0$.
\end{prop}

Let us briefly sketch the proof of Proposition~\ref{p:controldiv}. A key role is played by the identity
	\begin{align*}
	\dv w &= -2w_1w_2\dv\Sigma_1(w) \\
	&\quad +(w_1^2-w_2^2)\dv\Sigma_2(w) \\
	&\quad +L(w)[\nabla w] (1-\abs{w}^2)\qquad\text{for }w\colon\R^2\to\R^2\text{ smooth,}
	\end{align*}
	where $L(w)$ is a linear form on $\R^{2\times 2}$ which depends smoothly on $w$. After applying this identity to the mollified map $w=m_\e$, we show that the right-hand side is a sum of terms which vanish as $\e\to 0$ in $\mathcal D'(\Omega)$, and terms which are bounded pointwise by $\mathcal P_\e$. Upon proving that $\mathcal P_\e$ is bounded in $L^{\frac 43}_{\loc}$, we are then able to conclude. These facts are a consequence of a preliminary regularity estimate: our map $m$ enjoys some Besov regularity, namely $m\in B^{\frac 13}_{4,\infty,\loc}$. This was proved in \cite{LP} and we recall it in Lemma~\ref{L1}. This Besov regularity directly implies  that $\mathcal P_\e$ is bounded in $L^{\frac 43}_{\loc}$, as shown in Lemma~\ref{L13}. Then the bounds on all above mentioned terms rely on pointwise commutator estimates. Specifically, we prove in Lemma~\ref{L5} that
	\begin{equation*}
	\left\vert \left[\Pi(m)\right]_\e -\Pi(m_\e)\right\vert \left\vert \nabla m_\e\right\vert \lesssim \mathcal P_\e(x)\qquad\text{for smooth maps }\Pi.
	\end{equation*}
	Choosing $\Pi=\abs{\cdot}^2$, this obviously enables us to estimate $L(m_\e)[\nabla m_\e](1-\abs{m_\e}^2)$. Moreover, remarking that the condition $\dv\Sigma_j(m)=0$ ensures that
	\begin{equation*}
	\dv\Sigma_j(m_\e) = \dv \left[ (\Sigma_j(m_\e)) -[\Sigma_j(m)]_\e \right],
	\end{equation*}
	we are also able to deal with the other terms, which are of the form
	\begin{align*}
	F_j(m_\e)\dv\Sigma_j(m_\e) &= \dv\left[ (F_j(m_\e) \left( (\Sigma_j(m_\e)) -[\Sigma_j(m)]_\e \right) \right] \\
	&\quad - DF_j(m_\e)[\nabla m_\e] \left((\Sigma_j(m_\e)) -[\Sigma_j(m)]_\e\right),
	\end{align*}
	for some smooth $F_j\colon \R^2\to \R$.
	The first term is easily seen to go to zero in $\mathcal D'(\Omega)$, while the second term is bounded by $\mathcal P_\e$ thanks to the pointwise commutator estimate with $\Pi=\Sigma_j$.
	
	Before turning to the full proof of Proposition~\ref{p:controldiv} we gather the intermediate results Lemmas~\ref{L1}, \ref{L13} and \ref{L5} below. In Lemma~\ref{L1} we recall from \cite{LP} the regularity $m\in B^{\frac 13}_{4,\infty,\loc}$. In Lemma~\ref{L13} we infer from this that $\PPI_\e$ is bounded in $L^{\frac 43}_{\loc}$. And in Lemma~\ref{L5} we establish  the above mentioned pointwise commutator estimates.

%
%

\begin{lem}[{\cite[Theorem~4]{LP}}]
	\label{L1} 
Let $m:\Omega\rightarrow\R^2$ satisfy \eqref{eq12}, then
	$m\in  B^{\frac{1}{3}}_{4,\infty, \loc}(\Omega)$.
\end{lem}
\begin{proof}[Proof of Lemma~\ref{L1}] This is essentially proved in \cite[Theorem~4]{LP}. For the reader's convenience we reproduce the argument here, adopting, for the sake of variety, a slightly different point of view. 
	
Since $\dv\Sigma_j(m)=0$, we infer that $\mathrm{curl}\lt(i\Sigma_j(m)\rt)=0$. Recall that from Remark \ref{rmk3} we are assuming $\Omega$ is simply connected, and thus there exists $F_j:\Omega\rightarrow \R$ with 
	\begin{equation}\label{eq13}
	\nabla F_j=i\Sigma_j(m)\qd\text{a.e.} 
	\end{equation}
 Since $\abs{m}=1$ a.e. we may choose $\theta\colon\Omega\to \R$ such that $m=e^{i\theta}$ a.e., and by definition \eqref{eq1}-\eqref{eq2} of $P$ and $K$  it follows that $F=(F_1,F_2)$ satisfies
\begin{equation*}
DF=\left( \begin{array}{c}
		i\Sigma_1(m) \\
		i\Sigma_2(m)
		\end{array}\right) =P(\theta)\in K \qquad\text{a.e.}
\end{equation*}

For any given $U\subset\subset\Omega$ and $h\in\R^2$ with $|h|$ sufficiently small, e.g. $|h|<\frac{1}{3}\mathrm{dist}(U,\partial\Omega)$, by \cite[Lemma~7]{LP} we have
\begin{equation*}
\det (DF(x+h)-DF(x))\geq c_0 \abs{DF(x+h)-DF(x)}^4
\end{equation*}
for some constant $c_0>0$ and a.e. $x\in\Omega$ with $\dist(x,\partial\Omega)>|h|$. By definition of $F$ in \eqref{eq13} we have
	\begin{equation*}
	\det (DF(\cdot +h)-DF) = (iD^h\Sigma_1(m))\cdot ( D^h\Sigma_2(m))\overset{\eqref{eq13}}{ =}D^h \nabla F_1\cdot D^h\Sigma_2(m),
	\end{equation*}
	and also
	\begin{equation*}
	\abs{D^h m(x)}\lesssim \abs{DF(x+h)-DF(x)}.
	\end{equation*}
	Hence gathering the three above equations, we obtain
	\begin{equation}\label{eq:detcontrolsDhm}
	\abs{D^h m}^4\lesssim D^h \nabla F_1\cdot D^h\Sigma_2(m) \qd\text{ for a.e. }x\in\Omega\text{ with }\mathrm{dist}(x,\partial\Omega)>|h|.
	\end{equation}
	Let $\eta\in C_c^\infty(\Omega)$ be a test function with $\mathrm{dist}(\supp\eta,\partial\Omega)>2h$ and $\one_U\leq\eta\leq\one_{\Omega}$. Integrating by parts and using that $\dv\Sigma_2(m)=0$ (and thus $\int_{\Omega} \na \lt(D^h F_1 \eta^2 \rt)\cdot D^h \Sigma_2(m) dx=0$), we have
	\begin{align*}
	\int_\Omega \eta^2 D^h \nabla F_1 \cdot D^h\Sigma_2(m) \, dx
	&=- \int_\Omega D^h F_1  D^h\Sigma_2(m) \cdot \nabla (\eta^2)\, dx\\
	&\lesssim \abs{h}\norm{\na F_1}_{L^\infty(\Omega)} \norm{\na\Sigma_2}_{L^\infty(\overline B_1)} \norm{\nabla\eta}_{L^\infty(\Omega)} \int_\Omega \abs{\eta}\abs{D^h m}\,dx \\
	&\lesssim \abs{h}\norm{\nabla\eta}_{L^\infty(\Omega)} \left(\int_\Omega \eta^2 \abs{D^h m}^4 \, dx\right)^{\frac 14}.
	\end{align*}
	Recalling \eqref{eq:detcontrolsDhm} we deduce
	\begin{equation*}
	\int_\Omega \eta^2\abs{D^h m}^4\, dx \lesssim \abs{h}\norm{\nabla\eta}_{L^\infty(\Omega)} \left(\int_\Omega \eta^2 \abs{D^h m}^4 \, dx\right)^{\frac 14},
	\end{equation*}
	and thus
	\begin{equation*}
	\left(\int_\Omega \eta^2 \abs{D^h m}^4 \, dx\right)^{\frac 14}\lesssim \abs{h}^{\frac 13}\norm{\nabla\eta}_{L^\infty(\Omega)}^{\frac 13}.
	\end{equation*}
	As $\one_U\leq\eta$,  it follows that
	\begin{equation*}
t^{-\frac{1}{3}}\sup_{\abs{h}\leq t}\norm{D^h m}_{L^4(U)}\lesssim\norm{\nabla\eta}_{L^\infty(\Omega)}^{\frac 13}
	\end{equation*}
for $t$ sufficiently small. For larger $t$ values, the boundedness of $m$ implies  that $t^{-\frac{1}{3}}\sup_{\abs{h}\leq t}\norm{D^h m}_{L^4(U)}$ is bounded. Thus $m\in B^{\frac 13}_{4,\infty}(U)$ for all $U\subset\subset\Omega$.
\end{proof}

%
%

\begin{lem} 
	\label{L13}
Given $m\in B^{\frac 13}_{4,\infty,\loc}(\Omega)$, let $\PPI_\e$ be as in \eqref{eq600}. Then for any $U\subset\subset\Omega$ and $\e$ sufficiently small we have
	\begin{equation}
	\label{eq601}
	\norm{\PPI_\e}_{L^{\frac 43}(U)}\leq\abs{m}_{B^{\frac 13}_{4,\infty}(U)}^{3}.
	\end{equation}
	In particular $\mathcal P_\e$ is bounded in $L^{\frac 43}_{\loc}(\Omega)$.
\end{lem}
\begin{proof}[Proof of Lemma~\ref{L13}] Jensen's inequality implies
	\begin{align*}
	\int_{U} \lt(\PPI_{\ep}\rt)^{\frac{4}{3}} dx&\overset{\eqref{eq600}}{=} \ep^{-\frac{4}{3}}\int_{U} \lt(\Xint{-}_{B_{\ep}(0)} \lt|D^z m(x)\rt|^3 dz\rt)^{\frac{4}{3}}  dx\nn\\
	&\leq \ep^{-\frac{4}{3}}\int_{U}  \Xint{-}_{B_{\ep}(0)} \lt|D^z m(x)\rt|^4 dz dx\nn\\
	&\leq \e^{-\frac 43} \sup_{\abs{z}\leq\e}\int_U \abs{D^z m(x)}^4\, dx,
	\end{align*}
	which gives \eqref{eq601}.
\end{proof}

%
%

\begin{lem}
	\label{L5} 
	Let $m\colon\Omega\to\R^2$ be such that $\abs{m}\leq R$ a.e. for some $0<R<\infty$ and $\Pi\in C^2(\R^2;\R)$. For any $x\in\Omega$ such that $\overline B_\e(x)\subset\Omega$, 
	we have
	\begin{align}
	\label{eq16}
	&\lt|\lt[\Pi(m)\rt]_{\ep}(x)-\Pi(m_{\ep}(x))\rt| \lt|\na m_{\ep}(x)\rt|   \lesssim \norm{\nabla^2\Pi}_{L^\infty(\overline B_R)} \PPI_{\ep}(x),
	\end{align}
	where $\PPI_\e$ is as in \eqref{eq600}.
\end{lem}

\begin{proof}[Proof of Lemma~\ref{L5}]
Let $x\in\Omega$ be such that $\overline B_{\ep}(x)\subset\Omega$. The commutator estimate \eqref{eqg2} proved in Appendix~\ref{a:commut}, Lemma~\ref{L2.5} gives
	\begin{equation*}
	\lt|\lt[\Pi(m)\rt]_{\ep}(x)-\Pi(m_{\ep}(x))\rt|\lesssim\norm{D^2\Pi}_{L^\infty(\overline B_R)} \Xint{-}_{B_\e(0)}\abs{D^z m(x)}^2\, dz,
	\end{equation*}
	so Jensen's inequality and the definition \eqref{eq600} of $\PPI_\e$ imply
	\begin{align}
	\lt|\lt[\Pi(m)\rt]_{\ep}(x)-\Pi(m_{\ep}(x))\rt|^{\frac 32}
	&\lesssim\norm{D^2\Pi}_{L^\infty(\overline B_R)}^{\frac 32} \Xint{-}_{B_\e(0)}\abs{D^z m(x)}^3\, dz \nn\\
	&\overset{\eqref{eq600}}{=}  \norm{D^2\Pi}_{L^\infty(\overline B_R)}^{\frac 32} \e \PPI_\e(x).
	\label{eq:commut32}
	\end{align}
	To estimate $\nabla m_\e$ we compute, using the fact that $\nabla\rho_\e$ has zero average,
	\begin{align*}
	\nabla m_\e(x) & =\int_{B_\e(0)} m(x-z)\nabla\rho_\e(z)\, dz 
	= \int_{B_\e(0)}D^{-z}m(x)\nabla\rho_\e(z)\, dz. 
	\end{align*}
	Hence by Jensen's inequality and \eqref{eq600} again,
	\begin{align}
	\abs{\nabla m_\e(x)}^3\lesssim \e^{-3}\Xint{-}_{B_\e(0)}\abs{D^{-z}m(x)}^3\, dz\,=\, \e^{-2}\PPI_\e(x).\label{eq:nablameps3}
	\end{align}
	From \eqref{eq:commut32}-\eqref{eq:nablameps3} we gather
	\begin{align*}
	\lt|\lt[\Pi(m)\rt]_{\ep}(x)-\Pi(m_{\ep}(x))\rt| \lt|\na m_{\ep}(x)\rt|  
	&\lesssim \left(\norm{D^2\Pi}_{L^\infty(\overline B_R)}^{\frac 32} \e \PPI_\e(x) \right)^{\frac 23}\left( \e^{-2}\PPI_\e(x)\right)^{\frac 13} \\
	&\lesssim \norm{D^2\Pi}_{L^\infty(\overline B_R)} \PPI_\e(x).
	\end{align*}
\end{proof}


Equipped with Lemmas~\ref{L1}, \ref{L13} and \ref{L5} we can now prove Proposition~\ref{p:controldiv}.

\begin{proof}[Proof of Proposition~\ref{p:controldiv}] 
	First note that by Lemmas~\ref{L1} and \ref{L13} we have $m\in B^{\frac 13}_{4,\infty,\loc}(\Omega)$ and $\mathcal P_\e$ is bounded in $L^{\frac 43}_{\loc}(\Omega)$, hence it admits weakly converging subsequences.
	
	Then the strategy of the proof involves the convoluted map $m_\e$, for which we can use the chain rule to compute $\dv\Sigma_k(m_\e)$. Using some algebraic identities specific to the Jin-Kohn entropies and the commutator estimates of Lemma~\ref{L5} this enables us to control $\dv m_\e$ in terms of $\PPI_\e$.
	
	For any smooth function $w:\Omega\rightarrow \R^2$,
	\begin{align}
	\dv \Sigma_1(w)&\overset{\eqref{eqx10}}{=}\partial_1 w_{2}\lt(1-w_1^2-\frac{w_2^2}{3}\rt)+w_2\lt(-2 w_1 \partial_1 w_{1}-\frac{2}{3} w_2 \partial_1 w_{2}\rt)\nn\\
	&\qd+\partial_2 w_{1}\lt(1-w_2^2-\frac{w_1^2}{3}\rt)+w_1\lt(-2 w_2 \partial_2 w_{2}-\frac{2}{3} w_1 \partial_2 w_{1}\rt)\nn\\
	&=-2 w_1 w_2 \dv w+\lt(\partial_1 w_2+\partial_2 w_1\rt)\lt(1-\lt|w\rt|^2\rt)\label{eqx11},\\
	\dv \Sigma_2(w)&\overset{\eqref{eqx10}}{=}-\partial_1 w_1\lt(1-\frac{2}{3}w_1^2\rt)-w_1\lt(-\frac{4}{3}w_1 \partial_{1} w_1\rt)\nn\\
	&\qd\qd +\partial_2 w_2\lt(1-\frac{2}{3}w_2^2\rt)+w_2\lt(-\frac{4}{3}w_2 \partial_2 w_2\rt)\nn\\
	&=-\partial_1 w_1\lt(1-2 w_1^2\rt)+\partial_2 w_2\lt(1-2 w_2^2\rt)\nn\\
	&=-\partial_1 w_1\lt(1-2 w_1^2\rt)+\lt(\dv w- \partial_1 w_1\rt)\lt(1-2 w_2^2\rt)\nn\\
	&=-\partial_1 w_1\lt(2-2 w_1^2-2w_2^2\rt)+\dv w \lt(1-2 w_2^2\rt)\nn\\
	&=\lt(w_1^2-w_2^2\rt) \dv w+\lt(\partial_2 w_2- \partial_1 w_1 \rt)\lt(1-\lt|w\rt|^2\rt).\label{eqx12}
	\end{align}
Multiplying \eqref{eqx11} by $-2w_1w_2$, \eqref{eqx12} by $(w_1^2-w_2^2)$ and adding the resulting identities, we infer
\begin{align*}
	\dv w &= G_1(w)\dv\Sigma_1(w)  +G_2(w)\dv\Sigma_2(w)+L(w)[\nabla w] (1-\abs{w}^2),
\end{align*}
where
\begin{align*}
G_1(w)&=-2w_1w_2,\qquad G_2(w)=w_1^2-w_2^2,
\end{align*}
and
\begin{align*}
L(w)[\nabla w]&=(1+\abs{w}^2)\dv w + 2 w_1w_2 (\partial_1w_2+\partial_2 w_1) \\
&\quad +(w_2^2-w_1^2)(\partial_2w_2-\partial_1 w_1).
\end{align*}	
Note that $L(w)$ is a linear form on $\R^{2\times 2}$ which depends smoothly on $w$, and that $G_1$, $G_2$ are smooth functions of $w$.
We fix $\Omega'\subset\subset\Omega$ and apply this to $w=m_\e$. Thus for small enough $\e$, 
\begin{align}
\dv m_\e &= G_1(m_\e)\dv\Sigma_1(m_\e)  +G_2(m_\e)\dv\Sigma_2(m_\e) \nn\\
	&\quad +L(m_\e)[\nabla m_\e] (1-\abs{m_\e}^2) \qd\qd\text{ in }\Omega'.\label{eq:divmeps1}
\end{align}
Recalling that $\dv\Sigma_k(m)=0$ for $k\in\lbrace 1,2\rbrace$, we also have
\begin{equation*}
\dv \left( [\Sigma_k(m)]_\e\right) =\left[\dv\Sigma_k(m)\right]_\e =0\qd\text{ in }\Omega',
\end{equation*}
and may therefore (using also that $\abs{m}=1$ a.e.) rewrite \eqref{eq:divmeps1} as
\begin{align}
\dv m_\e & =A_1^\e +A_2^\e+B^\e +R_1^\e+R_2^\e,\label{eq:divmeps}
\end{align}
where
\begin{align}
A_k^\e &=  - DG_k(m_\e)[\nabla m_\e]\cdot\left( \Sigma_k(m_\e)- [\Sigma_k(m)]_\e \right)  \qquad\text{for }k=1,2,\nn\\
B^\e&=L(m_\e)[\nabla m_\e] (1-\abs{m_\e}^2) = L(m_\e)[\nabla m_\e] ([\abs{m}^2]_\e -\abs{m_\e}^2),  \nn\\
R_k^\e&= \dv \left[ G_k(m_\e)\left(\Sigma_k(m_\e)-[\Sigma_k(m)]_\e\right)\right]\qquad\text{for }k=1,2.\nn
\end{align}
For $k\in\lbrace 1,2\rbrace$, because $\Sigma_k$ is smooth and $m\in L^\infty(\Omega)$ 
we have $\Sigma_k(m_\e)\to \Sigma_k(m)$ and $[\Sigma_k(m)]_\e\to \Sigma_k(m)$ strongly in $L^p(\Omega')$ for any $p\in [1,\infty)$, and in particular
\begin{align*}
&\Sigma_k(m_\e)-[\Sigma_k(m)]_\e \\
&=\Sigma_k(m_\e)-\Sigma_k(m) + \Sigma_k(m)- [\Sigma_k(m)]_\e
\longrightarrow 0\qquad\text{strongly in }L^1(\Omega').
\end{align*}
Since $\abs{m_\e}\leq 1$ and $G_k$ is smooth this implies
\begin{equation}\label{eq:Rkto0}
R_k^\e = \dv \left[ G_k(m_\e)\left(\Sigma_k(m_\e)-[\Sigma_k(m)]_\e\right)\right] \longrightarrow 0\qquad\text{in }\mathcal D'(\Omega').
\end{equation}
Next we notice that  $A^\e_1 + A^\e_2+ B^\e$ in \eqref{eq:divmeps} is a sum  of terms of the form
\begin{equation*}
X^\e =T(m_\e)[\nabla m_\e]\left( \left[\Pi(m)\right]_\e -\Pi(m_\e)\right),
\end{equation*}
for some smooth functions $\Pi$ and linear forms $T(m_\e)$ depending smoothly on $m_\e$. Recalling again that $\abs{m_\e}\leq 1$, Lemma~\ref{L5} therefore ensures that
\begin{equation}\label{eq:boundAB}
\abs{A^\e_1 + A^\e_2+ B^\e}\lesssim \mathcal P_\e\qquad\text{in }\Omega'.
\end{equation}
Plugging \eqref{eq:boundAB} and \eqref{eq:Rkto0} into \eqref{eq:divmeps} and recalling that $m_\e\to m$ in $L^1(\Omega')$ we infer
\begin{align}
\left\vert \int_\Omega m\cdot\nabla \zeta \, dx \right\vert &= \liminf_{\e\to 0} \left\vert \int_\Omega m_{\e}\cdot\nabla \zeta \, dx \right\vert
 \nn\\
&
 \lesssim \liminf_{\e\to 0} \int_\Omega \abs{\zeta}\mathcal P_{\e}\, dx\qquad\text{for all }\zeta\in C_c^\infty(\Omega').\label{eq:bounddivm}
\end{align}

Thanks to Lemma~\ref{L13} we may choose a sequence $\e_n\to 0$ such that $\PPI_{\e_n}\rightharpoonup \PPI$ weakly in $L^{\frac 43}_{\loc}$.
Then \eqref{eq:bounddivm} gives
	\begin{equation*}
	\lt|\int_{\Omega} m  \cdot \na \zeta dx\rt|\lesssim  \int_{\Omega}  \PPI \lt|\zeta\rt| \; dx.
	\end{equation*}
	In particular
	\begin{equation*}
	\lt|\int_{\Omega} m  \cdot \na \zeta dx\rt|\lesssim \|\zeta\|_{ L^{4}(\Omega)}\|\PPI\|_{ L^{\frac{4}{3}}(\supp\zeta)}\qquad\text{ for any }\zeta\in C_c^{\infty}(\Omega),
	\end{equation*}
	which implies that $\dv m\in L^{\frac{4}{3}}_{\loc}(\Omega)$. Taking $\zeta=\chara_{U}*\rho_{\delta}$ and letting $\delta\rightarrow 0$ gives 
	\begin{equation*}
	\lt|\int_{U} \dv m \, dx\rt|\lesssim  \int_{\bar{U}}  \PPI dx\qd\text{ for all }U\subset \subset \Omega.
	\end{equation*}
	Now choosing $U=B_r(x)$ and letting $r\rightarrow 0$ we get \eqref{eq340} for all Lebesgue points of $\dv m$ and $\PPI$.  \end{proof}

\section{Control of  all entropy productions}\label{s:controlent}

Recall that an entropy is a smooth map $\Phi:\mathbb{S}^1\rightarrow \R^2$ that satisfies $e^{it}\cdot \frac{d}{dt}\Phi(e^{it})=0$ for all $t\in\R$.

\begin{prop}\label{p:controlent} 
	Let $m:\Omega\rightarrow\R^2$ satisfy \eqref{eq12}. 
 For any entropy $\Phi\in C^2(\mathbb{S}^1)$ we have $\dv\Phi(m)\in  L^{\frac 43}_{\loc}(\Omega)$ and
	\begin{equation}
	\label{eq621}
\abs{\dv\Phi(m)(x)}\lesssim \norm{\Phi}_{C^2(\mathbb{S}^1)}\mathcal P(x)\qquad
\text{for a.e. }x\in \Omega,
	\end{equation}
where $\PPI\in L^{\frac 43}_{\loc}(\Omega)$ is as in Proposition~\ref{p:controldiv}.
\end{prop}

	The proof of Proposition~\ref{p:controlent} is a refinement of the argument in \cite[Proposition~3]{DeI}, replacing their commutator estimates by the pointwise commutator estimates proved in Lemma~\ref{L5}, Section~\ref{s:controldiv}.

\begin{proof}[Proof of Proposition~\ref{p:controlent}] Let $\eta:\lt[0,\infty\rt)\rightarrow \R$ be a cut-off function with 
	$\eta=0$ on $\lt[0,\frac{1}{2}\rt]\cup \lt[2,\infty\rt)$, $\eta(1)=1$, and define
	\begin{equation*}
	\wt{\Phi}(z):=\eta\lt(\lt|z\rt|\rt) \Phi\lt(\frac{z}{\lt|z\rt|}\rt).
	\end{equation*}
	By Step 2 of the proof of \cite[Proposition~3]{DeI} we have that 
	\begin{equation*}
D\wt{\Phi}(z)=-2 \Psi(z)\otimes z+\gamma(z) Id\quad\text{ for every }z\in \R^2
	\end{equation*}
	where $Id$ is the identity matrix and 
	\begin{equation}
	\label{eq305}
	\gamma(z)=\frac{z^{\perp}\cdot D\wt{\Phi}(z) z^{\perp}}{\lt|z\rt|^2}\text{ and }
	\Psi(z)=\frac{-D\wt{\Phi}(z) z+\gamma(z)z}{2 \lt|z\rt|^2}.
	\end{equation}
	Note that 
	\begin{align}
	\label{eq311}
\lVert \gamma\rVert_{W^{1,\infty}(\R^2)}&\lesssim \lt(\|D^2\wt{\Phi}\|_{L^{\infty}(\R^2)}+\|D\wt{\Phi}\|_{L^{\infty}(\R^2)}\rt) \lesssim \norm{\Phi}_{C^2(\mathbb{S}^1)},
	\end{align}
	and hence 
	\begin{equation}
	\label{eq314}
	\|D\Psi\|_{L^{\infty}(\R^2)}\overset{\eqref{eq311}, \eqref{eq305}}{\lesssim} \norm{\Phi}_{C^2(\mathbb{S}^1)}.
	\end{equation}
	Exactly as in Step 3 of the proof of  \cite[Proposition~3]{DeI} we see that 
	\begin{equation}
	\label{eq306}
	\dv \wt{\Phi}(m_{\ep})=\Psi\lt(m_{\ep}\rt)\cdot \na \lt(1-\lt|m_{\ep}\rt|^2\rt)+\gamma\lt(m_{\ep}\rt) \dv m_{\ep}.
	\end{equation}
	Testing with $\zeta\in C^{\infty}_c(\Omega)$ we have 
	\begin{align}
	\label{eq685}
	&\int_{\Omega} \wt{\Phi}(m_{\ep})\cdot \na \zeta dx\nn\\
	&\overset{\eqref{eq306}}{=}\int_{\Omega} \lt(1-\lt|m_{\ep}\rt|^2\rt)\Psi(m_{\ep})\cdot \na \zeta dx \nn\\
	&\qd +\int_{\Omega} \lt(\Psi_{1,1}(m_{\ep})m_{\ep 1,1}+ \Psi_{1,2}(m_{\ep})m_{\ep 2,1}\rt. \nn\\
	&\qd\qd\qd\qd\qd\qd\qd\qd\qd\lt.    +\Psi_{2,1}(m_{\ep})m_{\ep 1,2} + \Psi_{2,2}(m_{\ep})m_{\ep 2,2}\rt)\lt(1-\lt|m_{\ep}\rt|^2\rt)\zeta\; dx\nn\\
	&\qd -\int_{\Omega} \gamma\lt(m_{\ep}\rt) \dv m_{\ep} \zeta \; dx.
	\end{align}
	Note that by Proposition \ref{p:controldiv}, $\dv m_{\e}\rightarrow\dv m$ in $L^{\frac{4}{3}}_{\loc}(\Omega)$. Thus, choosing  a sequence $\e_n$ such that $\PPI_{\e_n}\rightharpoonup \PPI$ weakly in $L^{\frac 43}_{\loc}$, we obtain
	\begin{align*}
	&\lt|\int_{\Omega}\Phi(m)\cdot\na\zeta dx\rt|=\lim_{n\rightarrow \infty}\lt|\int_{\Omega} \wt{\Phi}(m_{\ep_n})\cdot \na \zeta dx\rt|\nn\\
	&  \overset{\eqref{eq685}, \eqref{eq311}, \eqref{eq314}}{\lesssim} \norm{\Phi}_{C^2(\mathbb{S}^1)} \lim_{n\rightarrow \infty} \int_{\Omega} \lt|\na m_{\ep_n}\rt|\lt| 1-\lt|m_{\ep_n}\rt|^2\rt| \lt|\zeta\rt|\; dx\nn\\
	&\qd\qd\qd\qd\qd +\norm{\Phi}_{C^2(\mathbb{S}^1)}\lim_{n\rightarrow \infty} \int_{\Omega}  \lt|\dv m_{\ep_n}\rt| \lt|\zeta\rt| \; dx\nn\\
	& \overset{\eqref{eq16}, \eqref{eq340}}{\lesssim}  \norm{\Phi}_{C^2(\mathbb{S}^1)}  \int_{\Omega} \PPI|\zeta| \; dx\qd
	\text{ for all }\zeta\in C^{\infty}_c(\Omega).
	\end{align*}
	Since $\mathcal P\in L^{\frac 43}_{\loc}(\Omega)$ this implies that $\dv \Phi(m)\in L^{\frac{4}{3}}_{\loc}(\Omega)$. Further \eqref{eq621} follows from the above and the same arguments presented at the end of the proof of Proposition~\ref{p:controldiv}. 
\end{proof}

\begin{rem}\label{r:controlent}
Later on we are going to specialize to entropies $\Phi$ of the form
\begin{align*}
\Phi(e^{it})=\psi(t)e^{it} +\psi'(t)ie^{it},
\end{align*}
for some $\psi\in C^3(\mathbb S^1)\approx C^3(\mathbb R/2\pi\mathbb Z)$.
Then the control established in Proposition~\ref{p:controlent} becomes
\begin{align}\label{eq622}
\abs{\dv\Phi(m)(x)}\lesssim \mathcal P(x)\norm{\psi}_{C^3(\mathbb S^1)}\qquad\text{for a.e. }x\in\Omega.
\end{align}
\end{rem}

\section{Computation of harmonic entropy productions}\label{s:compent}

In \cite{mul2}, entropies were first defined as smooth maps $\Phi:\R^2\rightarrow\R^2$ that satisfy  $e^{it}\cdot \frac{d}{dt}\Phi(re^{it})=0$ for all $t\in\R$ and $r>0$. Such entropies can be obtained from smooth functions $\varphi$ via the formula
\begin{equation}\label{eq15}
\Phi^{\varphi}(z)=\varphi(z)z +((iz)\cdot\nabla\varphi(z))iz\qquad\forall z\in\R^2.
\end{equation}
In \cite{LP}, the second two authors introduced the notion of \emph{harmonic entropies}, which are entropies $\Phi$ given by harmonic functions $\varphi$ through \eqref{eq15}.  They enjoy nice factorization properties with respect to the Jin-Kohn entropies, and this fact was a major ingredient in \cite{LP}. 

 While the entropy production $\dv\Phi(m)$ only depends on the values of $\Phi$ on $\mathbb S^1$, for the purpose of  estimating $\dv\Phi(m_\e)$ one needs $\Phi$ to be extended outside $\mathbb S^1$ (as  in the proof of Proposition~\ref{p:controlent}). Since $\abs{m_\e}\leq 1$ it is however enough to specify values of $\Phi \in \overline B_1$ (rather than all of $\R^2$ as the entropies used in \cite{mul2,LP}).

 In this section, we use the nice factorization properties of harmonic entropies given by \eqref{eq15} for $\varphi\in C^3(\overline B_1)$ harmonic in $B_1$, and obtain an explicit pointwise expression for the associated entropy productions.

For any function $\psi\in L^2(\mathbb S^1)$ we denote by $E\psi$ its harmonic extension, that is,   $E\colon L^2(\mathbb S^1)\to L^2(B_1)$ is the continuous linear operator 
uniquely determined by its action on Fourier modes
\begin{align}\label{eq:E}
E\psi_k (re^{i\theta}) =r^{\abs{k}}e^{ik\theta},\qquad\text{for }\psi_k(\theta)=e^{ik\theta}.
\end{align}
In what follows we will use only the basic fact that $E$ is continuous from $C^4(\mathbb S^1)$ to $ C^3(\overline B_1)$ (see the proof of Lemma~\ref{l:Apsi}).

\begin{prop}\label{p:compent}
Let $m$ satisfy \eqref{eq12}.  
Given any $\psi\in C^4(\mathbb S^1)$ and $\varphi=E\psi$ its harmonic extension, the harmonic entropy given by
\begin{align*}
\Phi^{E\psi}(z)=\Phi^\varphi(z)=\varphi(z)z+((iz)\cdot \nabla\varphi(z))iz\qquad\forall z\in\overline B_1,
\end{align*}
satisfies
\begin{align}\label{eq:compent}
\dv\Phi^{E\psi}(m)(x)&=\mathcal A\psi(m(x))\dv m(x)\qquad\text{for a.e. }x\in\Omega,
\end{align}
where $\mathcal A\colon C^4(\mathbb S^1)\to C^0(\mathbb S^1)$ is the Fourier multiplier operator given by
\begin{align}
\label{eqreva12}
\mathcal A\psi_k =\frac{\abs{k}^3-2k^2-\abs{k}+2}{2}\psi_k,\qquad \psi_k(\theta)=e^{ik\theta}\text{ for all }k\in\mathbb Z.
\end{align}
\end{prop}

\begin{rem}\label{r:realvscomplex}
The Fourier multiplier operator $\mathcal A$ is naturally defined on complex-valued functions but in \eqref{eq:compent} only its restriction to real-valued functions plays a role.
\end{rem}

\begin{rem}\label{r:A}
The Fourier multiplier operator $\mathcal A$ defined by \eqref{eqreva12} is well-defined and continuous from $C^4(\mathbb S^1)$ to $C^0(\mathbb S^1)$. To show this, note that the Fourier coefficients of $\psi=\sum_k c_k(\psi)\psi_k\in C^4(\mathbb S^1)$ satisfy $k^4\abs{c_k(\psi)}=\abs{c_k(\psi^{(4)})}\in \ell^2$.
 Noting the fact that $\lt|c_k(\psi)\rt|\lt|\mathcal{A}(\psi_k)\rt|\lesssim \frac{1}{k} \lt|c_k(\psi)\rt| k^4 $ this implies, via Cauchy-Schwarz' inequality in $\ell^2$, that the Fourier
 series defining $\mathcal A\psi$ converges uniformly (so $\mathcal A$ is well-defined) and  gives the continuity estimate $\norm{\mathcal{A}\psi}_{C^0(\mathbb S^1)}\lesssim \abs{c_0(\psi)}+ \norm{\psi^{(4)}}_{L^2(\mathbb S^1)}\lesssim\norm{\psi}_{C^4(\mathbb S^1)}$. For finer boundedness properties of Fourier multiplier operators see e.g. \cite{grafakosclassical}.
\end{rem}

We split the proof of Proposition~\ref{p:compent} into Lemmas~\ref{l:compentharm} and \ref{l:Apsi}.
 The most crucial one is Lemma~\ref{l:compentharm}, where we rely on arguments from \cite{LP} to explicitly compute $\dv \Phi(m)$ for any harmonic entropy $\Phi$. 
Then in Lemma~\ref{l:Apsi}  we use this computation and the definition of the harmonic extension operator to obtain \eqref{eq:compent}.

\begin{lem}\label{l:compentharm}
Let $m$ satisfy \eqref{eq12}. 
	Let $\varphi\in  C^{3}(\overline B_1)$ be such that $\Delta\varphi=0$ in $B_1$ and $\Phi^{\varphi}$ be the corresponding harmonic entropy given by
	\begin{equation*}
\Phi^{\varphi}=\varphi(z)z+((iz)\cdot \nabla\varphi(z))iz\qquad\forall z\in\overline B_1.
	\end{equation*}
	Then we have
	\begin{align*}
	\dv\Phi^{\varphi}(m)& = A^\varphi(m) \dv m\qquad\text{a.e. in }\Omega,
	\end{align*}
	where $A^\varphi\in C^{0}(\overline B_1)$ is given by
	\begin{align}
	A^\varphi(z)&=\varphi(z) -z_1\partial_1\varphi(z) - z_2\partial_2\varphi(z) \nn\\
	&\quad + z_1z_2 \Big[ \partial_{12}\varphi(z)- z_2\partial_{111}\varphi(z) + z_1\partial_{211}\varphi(z)\Big] \nn\\
	&\quad +\frac 12 (z_1^2-z_2^2)\Big[ 
	\partial_{11}\varphi(z) + z_2\partial_{112}\varphi(z) + z_1\partial_{111}\varphi(z)
	\Big].\label{eq:Aphi}
	\end{align}
\end{lem}
\begin{proof}[Proof of Lemma~\ref{l:compentharm}]
	The convoluted map $m_\e$ is smooth with values into $\overline B_1$, and a direct computation (to be found in Appendix~\ref{a:compdivPhismooth}, Lemma~\ref{l:compdivPhismooth}) shows that
	\begin{align}
	\dv\Phi^{\varphi}(m_\e) &= A^\varphi(m_\e)\dv m_\e + R^0_\e + R^1_\e +R^2_\e,\label{eq:divPhimeps},
	\end{align}
	where
	\begin{align}
	R^0_\e & =\dv((\abs{m_\e}^2-1) B^{\varphi}(m_{\ep})), \nn \\
	R^j_\e&= F_j^{\varphi}(m_\e) \dv\Sigma_j(m_\e) \qquad \text{for }j=1,2,\nn
	\end{align}
	where $B^{\varphi} \colon \overline B_1\to\R^2$ and $F_1^{\varphi},F_2^{\varphi}\colon \overline B_1\to \R$ are continuous and depend only on $\varphi$. 
	
	Since $m\in L^\infty$ we have $m_\e\to m$ in $L^p(\Omega)$ for all $p\in [1,\infty)$. In particular, $\Phi^{\varphi}$ being 
$C^2$, this implies  that $\Phi^{\varphi}(m_\e)\to \Phi^{\varphi}(m)$ in $L^1(\Omega)$ and therefore
	\begin{equation*}
	\dv\Phi^{\varphi}(m_\e)\longrightarrow \dv\Phi^{\varphi}(m)\qquad\text{in }\mathcal D'(\Omega).
	\end{equation*}
	Recall that $\dv m\in L^{\frac 43}_{\loc}$ by Proposition~\ref{p:controldiv}, thus $\dv m_\e\to\dv m$ in $L^{\frac 43}_{\loc}$. Since $A^\varphi\colon\overline B_1\to\R$ is continuous, $m_\e\to m$ a.e. and $\abs{m_\e}\leq 1$, by dominated convergence we have $A^\varphi(m_\e)\to A^\varphi(m)$ in $L^4$ and we deduce
	\begin{equation*}
	A^\varphi(m_\e)\dv m_\e \longrightarrow A^\varphi(m)\dv m\qquad\text{in }L^1_{\loc}(\Omega)\text{ and hence in }\mathcal D'(\Omega).
	\end{equation*}
	Similarly we have $(\abs{m_\e}^2-1)B^{\varphi}(m_{\ep})\to 0$ in $L^1(\Omega)$, and
	\begin{equation*}
	R^0_\e \longrightarrow 0\qquad\text{in }\mathcal D'(\Omega).
	\end{equation*}
	Hence to conclude the proof of Lemma~\ref{l:compentharm} it suffices to show that
	\begin{equation}\label{eq:Rjepsto0}
	R^j_\e\longrightarrow 0\qquad\text{in }\mathcal D'(\Omega)\quad\text{for }j=1,2,
	\end{equation}
	to pass to the limit in \eqref{eq:divPhimeps} and to use $\dv m\in L^{\frac{4}{3}}_{\loc}(\Omega)$. 
	
	The proof of \eqref{eq:Rjepsto0} follows the ideas of \cite[Section~6]{LP}, with  slight modifications.  It relies on two crucial ingredients: the vanishing of the Jin-Kohn entropy productions  $\dv\Sigma_j(m)=0$ for $j=1,2$; and the regularity  $m\in B^{\frac 13}_{4,\infty,\loc}$, as used also in Section~\ref{s:controldiv}. 
	
	Let $j\in\lbrace 1,2\rbrace$. Using the explicit expression of $\dv\Sigma_j(m_\e)$ obtained from \eqref{eqx11}-\eqref{eqx12}, we have
	\begin{align*}
	\abs{\dv\Sigma_j(m_\e)}\lesssim \abs{\dv m_\e} + \abs{\nabla m_\e}(1-\abs{m_\e}^2).
	\end{align*}
	Recall from Proposition~\ref{p:controldiv} that $\dv m\in L^{\frac 43}_{\loc}$, and from Lemma~\ref{L5} (applied to $\Pi=\abs{\cdot}^2$) that
	\begin{align*}
	\abs{\nabla m_\e}(1-\abs{m_\e}^2)\lesssim \mathcal P_\e.
	\end{align*}
	Since $\mathcal P_\e$ is bounded in $L^{\frac 43}_{\loc}$ (see Lemma~\ref{L13}), we deduce from the above that  $\dv\Sigma_j(m_\e)$ is bounded in $L^{\frac 43}_{\loc}$. Because we also have $\dv\Sigma_j(m_\e)\to\dv\Sigma_j(m)=0$ in $\mathcal D'(\Omega)$, we infer
		\begin{equation*}
		\dv \Sigma_j(m_\e)\rightharpoonup 0\qquad \text{ in } L^{\frac{4}{3}}_{\loc}(\Omega).
		\end{equation*}
		Combined with the fact that $F_j^\varphi(m_\e)$ converges strongly to $F_j^\varphi(m)$ in $L^{4}_{\loc}(\Omega)$, this implies (invoking e.g. \cite[Proposition~3.5(iv)]{brezis})
		\begin{equation*}
		F_j^\varphi(m_\e)\dv \Sigma_j(m_\e)\rightharpoonup 0\qquad\text{in }L^1_{\loc}(\Omega),
		\end{equation*}
		which proves \eqref{eq:Rjepsto0}. 
\end{proof}

The second step towards Proposition~\ref{p:compent} is to obtain an expression of $A^{\varphi}$ in terms of the Fourier coefficients of $\varphi_{\lfloor \mathbb S^1}$. While Lemma~\ref{l:compentharm} naturally defines $A^\varphi$ for real-valued functions $\varphi$ (which give rise to $\R^2$-valued entropies $\Phi^\varphi$), formula \eqref{eq:Aphi} also makes sense for complex-valued functions and in what follows it is convenient to consider $A^\varphi$ to be extended to complex-valued functions $\varphi$ through \eqref{eq:Aphi}.

\begin{lem}\label{l:Apsi}
The operator
\begin{align}
\label{eq:defApsi}
\mathcal A\colon C^4(\mathbb S^1)\longrightarrow C^0(\mathbb S^1),\quad \psi\longmapsto A^{E\psi}_{\lfloor \mathbb S^1}
\end{align}
is the Fourier multiplier operator 
characterized by
\begin{align}\label{eq:Apsifourier}
\mathcal A\psi_k =\frac{\abs{k}^3-2k^2-\abs{k}+2}{2}\psi_k,\qquad \text{ for all }k\in\mathbb Z,
\end{align}
where $\psi_k(\theta)=e^{ik\theta}$.
\end{lem}
\begin{proof}[Proof of Lemma~\ref{l:Apsi}]
The linear operator $\mathcal A$ defined via \eqref{eq:defApsi} is continuous from $C^4(\mathbb S^1)$ to $C^0(\mathbb S^1)$ thanks to the bounds
\begin{align*}
\norm{\mathcal A\psi}_{C^0(\mathbb S^1)} \lesssim \norm{E\psi}_{C^3(\overline B_1)}\lesssim \norm{\psi}_{C^4(\mathbb S^1)}.
\end{align*}
The first inequality follows directly from the explicit definition of $A^\varphi$ \eqref{eq:Aphi}.
The second is a consequence of $\norm{ E\psi_k}_{C^3(\overline B_1)}\lesssim 1+\abs{k}^3$  and of the estimates in Remark~\ref{r:A}. Moreover, as shown in Remark~\ref{r:A}, the Fourier multiplier operator $\mathcal B$ defined by the right-hand side of \eqref{eq:Apsifourier} is continuous from $C^4(\mathbb S^1)$ to $C^0(\mathbb S^1)$. The two linear continuous operators $\mathcal A$ and $\mathcal B$ agree on $C^4(\mathbb S^1)$ provided they agree on the Fourier modes $\psi_k$, whose linear span is dense. Hence we only need to show \eqref{eq:Apsifourier} for each fixed $k\in\mathbb Z$.

For $k\geq 1$ the harmonic extension of $\psi_k$ is $E\psi_k=\varphi_k$ \eqref{eq:E} where $\varphi_k(z)=z^k$ (here we identify $\R^2$ and $\mathbb C$), so $\partial_1\varphi_k=k\varphi_{k-1}$ and $\partial_2\varphi_k=ik\varphi_{k-1}$. 
For $k\geq 3$ we may thus compute
\begin{align*}
A^{\varphi_k}&=\varphi_k -k\left(z_1\varphi_{k-1}+iz_2\varphi_{k-1}\right) \\
&\quad +z_1z_2\left(ik(k-1)\varphi_{k-2}-z_2k(k-1)(k-2)\varphi_{k-3}+iz_1 k(k-1)(k-2)\varphi_{k-3}\right)\\
&\quad +\frac 12 (z_1^2-z_2^2)\left(k(k-1)\varphi_{k-2}+iz_2k(k-1)(k-2)\varphi_{k-3}+z_1k(k-1)(k-2)\varphi_{k-3}\right) \\
& =(1-k)\varphi_k +iz_1z_2 k(k-1)^2\varphi_{k-2} +\frac 12 (z_1^2-z_2^2)k(k-1)^2\varphi_{k-2}.
\end{align*}
To obtain the last equality we used that
\begin{align*}
(z_1+iz_2)\varphi_{j-1}=z\varphi_{j-1}=\varphi_j\quad\text{for }j =k\text{ and }k-2.
\end{align*}
Computing further and using that $z^2\varphi_{k-2}=\varphi_k$ we find
\begin{align*}
A^{\varphi_k}&= (1-k)\varphi_k +\frac{1}{2}k(k-1)^2(z_1+iz_2)^2\varphi_{k-2} \\
&=\left(1 -k+\frac 12 k(k-1)^2\right)\varphi_k 
=\frac{k^3-2k^2-k+2}{2}\varphi_k.
\end{align*}
Since $\varphi_k=E\psi_k$ this proves \eqref{eq:Apsifourier} for $k\geq 3$. Similar computations show that \eqref{eq:Apsifourier} is valid for $0\leq k\leq 2$. For $k<0$ we obtain \eqref{eq:Apsifourier} by remarking that $\varphi_{k}$ is the complex conjugate of $\varphi_{-k}$ and therefore $A^{\varphi_k}$ is the complex conjugate of $A^{\varphi_{-k}}$.
\end{proof}

Proposition~\ref{p:compent} follows directly from Lemmas~\ref{l:compentharm} and \ref{l:Apsi}.

\section{Proof of Theorems~\ref{TT0} and \ref{TT1}}\label{s:proof}

	The proof of Theorem \ref{TT0} reduces to showing $\dv m=0$, which then allows us to invoke the main theorem in \cite{LP} to conclude the rigidity of $m$. To this end we compare, on the one hand the pointwise control 
	\begin{equation*}
	\lt|\dv\Phi^{E\psi}(m)(x)\rt|\lesssim\norm{\psi}_{C^3}\mathcal P(x),
	\end{equation*}
	obtained in Proposition~\ref{p:controlent} (see also Remark~\ref{r:controlent}), and on the other hand the explicit expression
	\begin{equation*}
	\dv\Phi^{E\psi}(m)(x)=\mathcal A\psi (m(x)) \dv m(x),
	\end{equation*}
	obtained in Proposition~\ref{p:compent} for some explicit Fourier multiplier operator $\mathcal A$. We show indeed that this operator can not satisfy a bound of the form $\vert\mathcal A\psi(z_0)\vert\lesssim \norm{\psi}_{C^3(\mathbb S^1)}$ for any $z_0\in\mathbb S^1$, because this would imply that the Hilbert transform $\mathcal H$ on the circle (or conjugate function transform), defined on Fourier modes $\psi_k=e^{ik\theta}$ by $\mathcal H\psi_k=-i\,\mathrm{sign}(k)\psi_k$, 
is bounded from $C^0(\mathbb S^1)$ into $L^\infty(\mathbb S^1)$, and this is known to be wrong.
 Therefore, the only way for the pointwise control of $\dv\Phi^{E\psi}(m)$ in terms of $\norm{\psi}_{C^3}$ to be valid is that $\dv m=0$. In addition to this basic argument, some technicalities enter the game due to the fact that the pointwise control and explicit expression of Propositions~\ref{p:controlent} and \ref{p:compent} are \enquote{almost everywhere} statements and the set of points $x$ at which they are valid depends in principle on $\psi$. We classically circumvent this by arguing on countable families of $\psi$ with appropriate density properties.

\begin{proof}[Proof of Theorem~\ref{TT0}]
	Let $\mathcal X \subset C^4(\mathbb S^1)$ denote a countable dense subset. Let $\mathcal G\subset\Omega$ be the set of all points $x\in\Omega$ at which  $\PPI(x)<\infty$, and both:
	\begin{itemize}
		\item the explicit expression 
		\begin{align*}
		\dv\Phi^{E\psi}(m)(x)=\mathcal A\psi(m(x))\dv m(x),
		\end{align*}
		given by \eqref{eq:compent},
		\item its pointwise control 
\begin{align*}
\abs{\dv\Phi^{E\psi}(m)(x)}\lesssim \mathcal P(x)\norm{\psi}_{C^3(\mathbb S^1)},
\end{align*}		
		 provided by \eqref{eq622},
	\end{itemize}
are satisfied  for all $\psi\in\mathcal X$. Thanks to Proposition~\ref{p:compent} and Proposition~\ref{p:controlent}, the set $\mathcal G$ is a countable intersection of sets of full measure, and therefore $|\Omega\setminus \mathcal G|=0$. 
	
	We claim that $\dv m(x)=0$ for all $x\in\mathcal G$. Assume by contradiction that $\dv m(x_0)\ne 0$ for some $x_0\in\mathcal G$. By definition of $\mathcal G$ we have
	\begin{align}\label{eq:estimx0}
	\abs{\mathcal A \psi(m(x_0)) \dv m (x_0)}&=\abs{\dv\Phi^{E\psi}(m)(x_0)}\leq C \norm{\psi}_{C^3(\mathbb S^1)},
	\end{align}
	for all $\psi\in\mathcal X$ and some constant $C=C(x_0,m)>0$. 
	
 In what follows we identify $\mathbb S^1$ with $\R/2\pi\mathbb Z$, functions on $\mathbb S^1$ with $2\pi$-periodic functions on $\mathbb R$, and $m(x_0)\in\mathbb S^1$ with its argument $\theta_0\in\mathbb R/2\pi\mathbb Z$.
 Dividing \eqref{eq:estimx0}  by $\dv m(x_0)\neq 0$ we deduce
	\begin{align*}
	\abs{\mathcal A\psi (\theta_0)}\leq \widetilde C \norm{\psi}_{C^3(\mathbb S^1)}
	\end{align*}
	for some other constant $\widetilde C=\widetilde C(x_0,m)>0$ and all $\psi\in\mathcal X$.
Since $\mathcal X$ is dense in $C^4(\mathbb S^1)$ and both sides of the above estimate depend continuously on $\psi$ in the $C^4$ topology, we infer that	
	\begin{align}
	\label{eq:estimAtheta0}
	\abs{\mathcal A\psi (\theta_0)}\leq \widetilde C \norm{\psi}_{C^3(\mathbb S^1)}\qquad\forall\psi\in C^4(\mathbb S^1).
	\end{align}
The multiplier operator $\mathcal A$ commutes with translations of the variable  and this automatically turns estimate \eqref{eq:estimAtheta0} at a fixed $\theta_0$ into an estimate at any $\theta\in\R$. Explicitly, for any $\psi\in C^4(\mathbb S^1)$ and $\varpi\in \R$ we may apply \eqref{eq:estimAtheta0} to the translated function $\psi^\varpi=\psi(\cdot +\varpi)$. Since $\mathcal A\psi^\varpi =\mathcal A \psi(\cdot +\varpi)$ (as can be checked on Fourier modes, see e.g. \cite[\S~3.6.1]{grafakosclassical}), we obtain $\abs{\mathcal A\psi(\theta_0+\varpi)}\leq \widetilde C \norm{\psi}_{C^3(\mathbb S^1)}$ for any $\psi$ and $\varpi$ and deduce
\begin{align}\label{eq:estimA}
\norm{\mathcal A\psi }_{L^\infty(\mathbb S^1)}\leq \widetilde C \norm{\psi}_{C^3(\mathbb S^1)}\qquad\forall\psi\in C^4(\mathbb S^1).
\end{align}
Splitting $\mathcal A=\frac 12 \mathcal A_0 +\frac 12 \mathcal A_1$ where
\begin{align*}
\mathcal A_0\psi_k =\abs{k}^3\psi_k,\qquad \text{and}\quad\mathcal A_1\psi_k =(2-2k^2-\abs{k})\psi_k,
\end{align*} 
we see (arguing as in Remark \ref{r:A}) that $\mathcal A_1$ defines a continuous linear operator from $C^3(\mathbb S^1)$ to $C^0(\mathbb S^1)$, and therefore \eqref{eq:estimA} implies that 
\begin{align}
\label{eq:estimA0}
\norm{\mathcal A_0\psi }_{L^\infty(\mathbb S^1)}\lesssim \norm{\psi}_{C^3(\mathbb S^1)}\qquad\forall \psi\in C^4(\mathbb S^1).
\end{align}
Remarking that
$\mathcal A_0\psi=-\mathcal H \psi^{(3)}$, where $\mathcal H$ is the Hilbert transform on the circle (or conjugate function transform, see \cite[\S~3.5]{grafakosclassical}), that is, the Fourier multiplier operator given by
\begin{align*}
\mathcal H\psi_0=0,\quad \mathcal H \psi_k =-i\,\mathrm{sign}(k)\psi_k\qquad\forall k\in\mathbb Z\setminus\lbrace 0\rbrace,
\end{align*}
we obtain the estimate
\begin{align*}
\norm{\mathcal H \psi^{(3)}}_{L^\infty(\mathbb S^1)}\overset{\eqref{eq:estimA0}}{\lesssim} \norm{\psi}_{C^3(\mathbb S^1)}\qquad\forall \psi\in C^4(\mathbb S^1),
\end{align*}
which readily turns into
\begin{align*}
\norm{\mathcal H\psi}_{L^\infty(\mathbb S^1)}\lesssim\norm{\psi}_{C^0(\mathbb S^1)}\qquad\forall \psi\in C^1(\mathbb S^1).
\end{align*}
Since $\mathcal H\colon C^0(\mathbb S^1)\to L^2(\mathbb S^1)$ is continuous and $C^1(\mathbb S^1)$ is dense in $C^0(\mathbb S^1)$, this estimate implies that $\mathcal H$ maps in fact $C^0(\mathbb S^1)$ continuously into $L^\infty(\mathbb S^1)$, which is notoriously wrong (see e.g. \cite[\S~VII.2]{zygmund}, the function $\psi(\theta)=\sum_{k\geq 2}\sin(k\theta)/(k\log k)$ is continuous but its Hilbert transform $\mathcal H\psi(\theta)=-\sum_{k\geq 2}\cos(k\theta)/(k\log k)$ is not bounded).

This concludes the proof that $\dv m=0$ a.e. Now $m$ indeed satisfies the assumptions of \cite[Theorem~3]{LP}, which gives the desired rigidity for $m$.
\end{proof}

The proof of Theorem \ref{TT1} relies on  the correspondence between solutions of the differential inclusion $DF\in K$ a.e. and unit vector fields $m$ satisfying the assumptions of Theorem \ref{TT0}.

\begin{proof}[Proof of Theorem \ref{TT1}] Let $F\colon \Omega\to \R^2$ be a Lipschitz map  such that $DF\in K$ a.e., and we define $m:\Omega\rightarrow\R^2$ by
		\begin{equation}
		\label{eq270}
m_1=-F_{1,1}-F_{2,2},\quad  m_2=F_{1,2}-F_{2,1}.
		\end{equation}
		For $x\in\Omega$ such that $DF(x)=P(\theta)\in K$, where recall that $P\colon\R\to K\subset\R^{2\times 2}$ is the parameterization of the set $K$ defined in \eqref{eq1}, it is clear from \eqref{eq270} and \eqref{eq1} that $m=e^{i\theta}$. Further 
		\begin{equation}
		\label{xxeqa28}
		i\Sigma_j(m(x))=\nabla F_j(x)\qquad\text{ for }j=1,2,
		\end{equation}
		where $\Sigma_1,\Sigma_2$ are the Jin-Kohn entropies defined in \eqref{eqx10}.
		Thus the vector field $m$ defined by \eqref{eq270} satisfies
	\begin{equation*}
|m(x)|=1 \text{ a.e.}\quad\text{and}\quad \dv\lt(\Sigma_j(m)\rt)\overset{\eqref{xxeqa28}}{=}\mathrm{curl} \lt(\na F_j\rt)=0\text{ for }j=1,2.
	\end{equation*}
By Theorem \ref{TT0}, $m$ is locally  Lipschitz  outside a locally finite set of points $S$. From \eqref{xxeqa28} we deduce that $DF$ agrees almost everywhere with a map $G$ that is locally  Lipschitz outside of $S$, and therefore $DF$ itself is locally  Lipschitz outside of $S$ (indeed outside of $S$ the map $G$ is locally the gradient of a function with Lipschitz derivatives, which has to agree with $F$ up to a constant). Moreover in any convex neighborhood of a point in $S$, $m$ is a vortex, which translates into \eqref{eq210}.
\end{proof}

Finally, Theorem \ref{TT2} is a reformulation of Theorem \ref{TT1} by identifying $v=F_1+iF_2$ as in the proof of Theorem 5 in \cite{LP}. Hence $v$ and $F$ have the same regularity.

\appendix

\section{A Commutator estimate}\label{a:commut}

In this appendix we prove the following basic commutator estimate:
\begin{lem}
	\label{L2.5}
	Given $\Pi\in C^{2}(\R^2)$ and $m\colon\Omega\to\R^2$ with $\abs{m}\leq R$ a.e. for some $0<R<\infty$, we have
	\begin{equation}\label{eqg2}
	|\lt[\Pi(m)\rt]_{\ep}(x)-\Pi(m_{\ep}(x))| \lesssim 
	\norm{D^2\Pi}_{L^\infty(\overline B_R)} \; 
	\Xint{-}_{B_{\ep}(0)}|D^z m(x)|^2\,dz.
	\end{equation}
\end{lem}

\begin{proof}[Proof of Lemma~\ref{L2.5}]
	The proof follows computations presented in \cite{titi} and recently in a context closer to ours also in \cite{DeI}.
	We write out
	\begin{align}
	\label{eqqoo1}
	&\lt[\Pi(m)\rt]_{\ep}(x) -\Pi(m_{\ep}(x))\nn\\
	& = \int \lt(\Pi(m(x-z))-\Pi(m_{\ep}(x))\rt)\rho_{\ep}(z)\,dz\nn\\
	&=\int D\Pi(m(x-z))\cdot\lt(m(x-z)-m_{\ep}(x)\rt)\rho_{\ep}(z)\,dz\nn\\
	&+\int \lt(\Pi(m(x-z))-\Pi(m_{\ep}(x)) - D\Pi(m(x-z))\cdot\lt(m(x-z)-m_{\ep}(x)\rt)\rt)\rho_{\ep}(z)\,dz\nn\\
	&=\int \lt(D\Pi(m(x-z))-D\Pi(m_{\ep}(x))\rt)\cdot\lt(m(x-z)-m_{\ep}(x)\rt)\rho_{\ep}(z)\,dz\nn\\
	&+\int \lt(\Pi(m(x-z))-\Pi(m_{\ep}(x)) - D\Pi(m(x-z))\cdot\lt(m(x-z)-m_{\ep}(x)\rt)\rt)\rho_{\ep}(z)\,dz.
	\end{align}
	By Taylor expansion we have
	\begin{align}
	&|\Pi(m(x-z))-\Pi(m_{\ep}(x)) - D\Pi(m(x-z))\cdot\lt(m(x-z)-m_{\ep}(x)\rt)| \nn\\
	&\qd\qd \lesssim \norm{D^2\Pi}_{L^\infty}|m(x-z)-m_{\ep}(x)|^2, \nn\\
	&\lt|\lt(D\Pi(m(x-z))-D\Pi(m_{\ep}(x))\rt)\cdot\lt(m(x-z)-m_{\ep}(x)\rt) \rt|\nn\\
	&\qd\qd\lesssim \norm{D^2\Pi}_{L^\infty} \lt|m(x-z)-m_{\ep}(x)\rt|^2, \nn
	\end{align}
	and plugging this into \eqref{eqqoo1},
	\begin{equation}\label{eqg3}
	|\lt[\Pi(m)\rt]_{\ep}(x)-\Pi(m_{\ep}(x))|\lesssim \norm{D^2\Pi}_{L^\infty}\int_{B_{\ep}(0)}|m(x-z)-m_{\ep}(x)|^2\rho_{\ep}(z)\,dz.
	\end{equation}
	Moreover by Jensen's inequality we have
	\begin{align*}
	&\int_{B_{\ep}(0)}|m(x-z)-m_{\ep}(x)|^2\rho_{\ep}(z)\,dz \\
	& = \int_{B_{\ep}(0)}\left\vert \int_{B_\e(0)} (m(x-z)-m(x-y))\rho_\e(y) dy\right\vert^2\rho_{\ep}(z)\,dz \\
	&\lesssim \int_{B_\e(0)}\int_{B_\e(0)}\abs{m(x-z)-m(x-y)}^2\rho_\e(y)\rho_\e(z)\, dydz\\
	& = \int_{B_\e(0)}\int_{B_\e(0)}\abs{D^{-z}m(x)-D^{-y}m(x)}^2\rho_\e(y)\rho_\e(z)\, dydz \\
	&\lesssim \Xint{-}_{\B_\e(0)}\abs{D^{-z}m(x)}^2\,dz + \Xint{-}_{\B_\e(0)}\abs{D^{-y}m(x)}^2\,dy \\
	&\lesssim \Xint{-}_{\B_\e(0)}\abs{D^{z}m(x)}^2\,dz.
	\end{align*}
	Plugging this estimate into \eqref{eqg3} gives \eqref{eqg2}.
\end{proof}

\section{Computations needed in the proof of Lemma~\ref{l:compentharm}}
\label{a:compdivPhismooth}

\begin{lem}\label{l:compdivPhismooth}
	Let $\varphi\in C^3(\overline B_1)$ such that $\Delta\varphi=0$ in $B_1$ and $\Phi^{\varphi}$ the corresponding harmonic entropy given by
	\begin{equation}\label{eq:harmPhi}
\Phi^{\varphi}(z)=\varphi(z)z+((iz)\cdot \nabla\varphi(z))iz\qquad\forall z\in\overline B_1.
	\end{equation}
	For any \underline{smooth} map $w\colon \Omega\to \overline B_1$ we have
	\begin{align*}
	\dv\Phi^{\varphi}(w)&=A(w)\dv w + \dv((\abs{w}^2-1)B(w)) \\
	&\quad + \partial_2 B_1(w)\dv\Sigma_1(w) -\partial_1 B_1(w) \dv \Sigma_2(w),
	\end{align*}
	where $A=A^\varphi\colon\overline B_1\to \R$ and $B=B^\varphi\colon \overline B_1\to\R^2$  are given by
	\begin{align}
	A^\varphi(z)&=\varphi(z) -z_1\partial_1\varphi(z) - z_2\partial_2\varphi(z) \nn\\
	&\quad + z_1z_2 \Big[ \partial_{12}\varphi(z)- z_2\partial_{111}\varphi(z) + z_1\partial_{211}\varphi(z)\Big] \nn\\
	&\quad +\frac 12 (z_1^2-z_2^2)\Big[ 
	\partial_{11}\varphi(z) + z_2\partial_{112}\varphi(z) + z_1\partial_{111}\varphi(z)
	\Big], \label{eq:Avarphi}\\
	B^\varphi(z)&=\left(
	\begin{array}{c}
	\partial_1\varphi(z) +\frac 12 z_2\partial_{12}\varphi(z)-\frac 12 z_1\partial_{22}\varphi(z)\\
	\partial_2\varphi(z) -\frac 12 z_2\partial_{11}\varphi(z)+\frac 12 z_1\partial_{12}\varphi(z)
	\end{array}
	\right).\label{eq:Bvarphi}
	\end{align}
\end{lem}
\begin{proof}[Proof of Lemma~\ref{l:compdivPhismooth}]
	We have
	\begin{equation*}
\Phi^{\varphi}(w)\overset{\eqref{eq:harmPhi}}{=}\lt(\begin{array}{c}  w_1 \varphi -w_2 \lt(-w_2 \partial_1 \varphi+w_1  \partial_2 \varphi  \rt)\\
	w_2 \varphi +w_1 \lt(-w_2 \partial_1 \varphi+w_1  \partial_2 \varphi  \rt)
	\end{array}\rt).
	\end{equation*}
	So 
	\begin{align}
	\label{eqx4.25}
	&\dv \Phi^{\varphi}(w)\nn\\
	&\qd= w_1 \partial_1 \varphi\partial_1 w_1+  w_1 \partial_2 \varphi\partial_1 w_2+\varphi\partial_1 w_1\nn\\
	&\qd\qd -\lt(-w_2  \partial_1 \varphi+ w_1 \partial_2 \varphi \rt) \partial_1 w_2\nn\\
	&\qd\qd -w_2\lt( -\partial_1 \varphi \partial_1 w_2-  w_2 \partial_{11} \varphi \partial_1 w_1- w_2 \partial_{12} \varphi \partial_1 w_2\rt.\nn\\
	&\qd\qd\qd\qd\qd  \lt.+\partial_2 \varphi \partial_1 w_1+  w_1 \partial_{12} \varphi \partial_1 w_1+ w_1 \partial_{22} \varphi \partial_1 w_2\rt)\nn\\
	&\qd\qd+ w_2 \partial_1 \varphi \partial_2 w_1+ w_2 \partial_2 \varphi \partial_2 w_2+\varphi \partial_2 w_2\nn\\
	&\qd\qd +\lt(-w_2  \partial_1 \varphi+ w_1 \partial_2 \varphi \rt) \partial_2 w_1\nn\\
	&\qd\qd + w_1\lt( -\partial_1 \varphi \partial_2 w_2-  w_2 \partial_{11} \varphi \partial_2 w_1- w_2 \partial_{12} \varphi \partial_2 w_2\rt.\nn\\
	&\qd\qd\qd\qd\qd  \lt.+\partial_2 \varphi \partial_2 w_1+ w_1 \partial_{12} \varphi \partial_2 w_1+ w_1 \partial_{22} \varphi \partial_2 w_2\rt)\nn\\
	&\qd= \lt(\varphi+w_1 \partial_1 \varphi- w_2 \partial_2 \varphi+w_2^2  \partial_{11} \varphi- w_1 w_2   \partial_{12} \varphi\rt) \partial_1 w_1\nn\\
	&\qd\qd + \lt(\varphi+w_2 \partial_2 \varphi- w_1 \partial_1 \varphi+w_1^2  \partial_{22} \varphi- w_1 w_2   \partial_{12} \varphi\rt) \partial_2 w_2\nn\\
	&\qd\qd +\lt(w_1 \partial_2 \varphi+ w_2 \partial_1 \varphi-w_1 \partial_2 \varphi+w_2^2  \partial_{12} \varphi- w_1 w_2   \partial_{22} \varphi+w_2 \partial_1 \varphi\rt) \partial_1 w_2\nn\\
	&\qd\qd +\lt(w_2 \partial_1 \varphi- w_2 \partial_1 \varphi+w_1 \partial_2 \varphi-w_1 w_2  \partial_{11} \varphi+ w_1^2   \partial_{12} \varphi+ w_1 \partial_2 \varphi\rt) \partial_2 w_1\nn\\
	&\qd= \lt(\varphi+w_1 \partial_1 \varphi- w_2 \partial_2 \varphi+w_2^2  \partial_{11} \varphi- w_1 w_2   \partial_{12} \varphi\rt) \partial_1 w_1\nn\\
	&\qd\qd + \lt(\varphi+w_2 \partial_2 \varphi- w_1 \partial_1 \varphi+w_1^2  \partial_{22} \varphi- w_1 w_2   \partial_{12} \varphi\rt) \partial_2 w_2\nn\\
	&\qd\qd +\lt(2\partial_1 \varphi+w_2  \partial_{12} \varphi- w_1  \partial_{22} \varphi\rt) w_2 \partial_1 w_2\nn\\
	&\qd\qd +\lt(2\partial_2 \varphi-w_2  \partial_{11} \varphi+ w_1  \partial_{12} \varphi\rt) w_1 \partial_2 w_1.
	\end{align}
	Noting that 
	\begin{equation*}
\partial_1 \lt(\frac{\lt|w\rt|^2}{2}\rt)-w_1 \partial_1 w_1=w_2 \partial_1 w_2\qd\text{ and }\qd\partial_2 \lt(\frac{\lt|w\rt|^2}{2}\rt)-w_{2} \partial_2 w_2=w_1 \partial_2 w_1,
	\end{equation*}
	and plugging this into \eqref{eqx4.25},
	\begin{align}
	\label{eqx7}
	&\dv \Phi^{\varphi}(w)\nn\\
	&\qd=  \lt(\varphi+w_1 \partial_1 \varphi- w_2 \partial_2 \varphi+w_2^2  \partial_{11} \varphi- w_1 w_2   \partial_{12} \varphi\rt.\nn\\
	&\qd\qd\qd\qd \lt. -2 w_1 \partial_1 \varphi-w_1 w_2   \partial_{12} \varphi+w_1^2 \partial_{22} \varphi\rt) \partial_1 w_1\nn\\
	&\qd\qd+  \lt(\varphi+w_2 \partial_2 \varphi- w_1 \partial_1 \varphi+w_1^2  \partial_{22} \varphi- w_1 w_2   \partial_{12} \varphi\rt.\nn\\
	&\qd\qd\qd\qd \lt. -2 w_2 \partial_2 \varphi-w_1 w_2   \partial_{12} \varphi+w_2^2 \partial_{11} \varphi\rt) \partial_2 w_2\nn\\
	&\qd\qd+ \lt(\partial_1 \varphi+\frac{1}{2} w_2 \partial_{12} \varphi-  \frac{1}{2} w_1 \partial_{22} \varphi \rt)\partial_1 \lt(\lt|w\rt|^2\rt)\nn\\
	&\qd\qd+ \lt(\partial_2 \varphi-\frac{1}{2} w_2 \partial_{11} \varphi+  \frac{1}{2} w_1 \partial_{12} \varphi \rt)\partial_2 \lt(\lt|w\rt|^2\rt)\nn\\
	&\qd =\lt(\varphi-w_1 \partial_1 \varphi-w_2 \partial_2 \varphi+w_1^2  \partial_{22} \varphi +w_2^2  \partial_{11} \varphi-2 w_1 w_2\partial_{12} \varphi  \rt) \dv w\nn\\
	&\qd\qd+ \lt(\partial_1 \varphi+\frac{1}{2} w_2 \partial_{12} \varphi-  \frac{1}{2} w_1 \partial_{22} \varphi \rt)\partial_1 \lt(\lt|w\rt|^2\rt)\nn\\
	&\qd\qd+ \lt(\partial_2 \varphi-\frac{1}{2} w_2 \partial_{11} \varphi+  \frac{1}{2} w_1 \partial_{12} \varphi \rt)\partial_2 \lt(\lt|w\rt|^2\rt) \nn\\
	& \qd= C(w)\dv w + B(w)\cdot\nabla(\abs{w}^2),
	\end{align}
	where $B=B^\varphi$ is as in \eqref{eq:Bvarphi} and
	\begin{equation}
	\label{eqx8}
	C(z)=\varphi-z_1 \partial_1 \varphi-z_2 \partial_2 \varphi+z_1^2  \partial_{22} \varphi +z_2^2  \partial_{11} \varphi-2 z_1 z_2\partial_{12} \varphi.
	\end{equation}
	We rewrite \eqref{eqx7} as 
	\begin{align}
	\label{eqx13}
	\dv  \Phi^{\varphi}(w)&=C(w) \dv w+ \dv \lt(\lt(\lt|w\rt|^2-1\rt)B(w)\rt)\nn\\
	&\qd+\lt(\partial_1 B_1 \partial_1 w_1+ \partial_2 B_1 \partial_1 w_2+ \partial_1 B_2 \partial_2 w_1+ \partial_2 B_2 \partial_2 w_2 \rt)\lt( 1-\lt|w\rt|^2\rt).
	\end{align}
	Note that 
	\begin{align}
	\partial_1 B_1&=\partial_{11} \varphi+\frac{1}{2} z_2 \partial_{112} \varphi- \frac{1}{2} \partial_{22}\varphi   -\frac{1}{2} z_{1}\partial_{221} \varphi,\label{eqx15}\\
	\partial_2 B_2&=\partial_{22} \varphi-\frac{1}{2} z_2 \partial_{112} \varphi- \frac{1}{2} \partial_{11} \varphi +\frac{1}{2} z_{1}\partial_{221} \varphi, \nn\\
	\text{so }\quad &
	\partial_1 B_1+\partial_2 B_2=\frac{1}{2}\Delta \varphi,\label{eqx17}
	\end{align}
	and
	\begin{align}
	\partial_2 B_1 &=\frac{3}{2}\partial_{12} \varphi+\frac{1}{2}z_2 \partial_{122} \varphi-\frac{1}{2}z_1 \partial_{222} \varphi, \label{eqx20}\\
	\partial_1 B_2&=\frac{3}{2}\partial_{12} \varphi-\frac{1}{2}z_2 \partial_{111} \varphi
	+\frac{1}{2}z_1 \partial_{112} \varphi,\nn\\
	\text{so }\quad &
	\partial_2 B_1-\partial_1 B_2=\frac{1}{2}z_2\lt(\partial_1 \Delta \varphi\rt)-\frac{1}{2}z_1\lt(\partial_2 \Delta \varphi\rt).\label{eqx24}
	\end{align}
	Since $\varphi$ is harmonic we have from \eqref{eqx17} and \eqref{eqx24} that 
	\begin{equation}
	\label{eqx23}
	\partial_1 B_1+\partial_2 B_2=0\qd\text{ and }\qd \partial_2 B_1-\partial_1 B_2=0.
	\end{equation}
	Recalling moreover the explicit expressions of $\dv\Sigma_j(w)$ computed in \eqref{eqx11}-\eqref{eqx12},
	we find that
	\eqref{eqx13} can be rewritten as 
	\begin{align}
	\dv \Phi^{\varphi}(w)&\overset{\eqref{eqx23}, \eqref{eqx13}}{=} C(w) \dv w+ \dv\lt(\lt(\lt|w\rt|^2-1\rt)B(w)\rt)\nn\\
	&\qd +\partial_1 B_1\lt(\partial_1 w_1-\partial_2 w_2\rt)\lt(1-\lt|w\rt|^2\rt)\nn\\
	&\qd+\partial_2 B_1\lt(\partial_1 w_2+\partial_2 w_1\rt)\lt(1-\lt|w\rt|^2\rt)\nn\\
	&\overset{\eqref{eqx11}-\eqref{eqx12}}{=}C(w) \dv w+ \dv\lt(\lt(\lt|w\rt|^2-1\rt)B(w)\rt)\nn\\
	&\qd-\partial_1 B_1\lt(\dv \Sigma_2(w)-(w_1^2-w_2^2)\dv w\rt)\nn\\
	&\qd+\partial_2 B_1\lt(\dv \Sigma_1(w)+2w_1 w_2\dv w\rt)\nn\\
	&=A(w)\dv w+\dv\lt(\lt(\lt|w\rt|^2-1\rt)B\rt)\nn\\
	&\quad +\partial_2 B_1 \dv \Sigma_1(w)-\partial_1 B_1 \dv \Sigma_2(w), \label{eqx25}
	\end{align}
	where 
	\begin{align*}
	A(w)&=C(w)+2\partial_2 B_1 w_1 w_2+\partial_1 B_1 \lt(w_1^2-w_2^2\rt)\nn\\
	&\overset{\eqref{eqx8}, \eqref{eqx20}, \eqref{eqx15}}{=}\varphi-w_1 \partial_1 \varphi-w_2 \partial_2 \varphi+w_1^2  \partial_{22} \varphi +w_2^2  \partial_{11} \varphi-2 w_1 w_2\partial_{12}\varphi\nn\\
	&\qd +2w_1 w_2\lt(\frac{3}{2}\partial_{12} \varphi+\frac{1}{2}w_2 \partial_{122} \varphi-\frac{1}{2}w_1 \partial_{222} \varphi\rt)\nn\\
	&\qd+\lt(w_1^2-w_2^2\rt)\lt(\partial_{11} \varphi+\frac{1}{2} w_2 \partial_{112} \varphi- \frac{1}{2} \partial_{22}\varphi   -\frac{1}{2} w_{1}\partial_{221} \varphi \rt)\nn\\
	&=\varphi-w_1 \partial_1 \varphi-w_2 \partial_2 \varphi\nn\\
	&\qd +2w_1 w_2\lt(\frac{1}{2}\partial_{12} \varphi+\frac{1}{2}w_2 \partial_{122} \varphi-\frac{1}{2}w_1 \partial_{222} \varphi\rt)\nn\\
	&\qd+\lt(w_1^2-w_2^2\rt)\lt(\frac{1}{2}\partial_{11} \varphi+\frac{1}{2} w_2 \partial_{112} \varphi  -\frac{1}{2} w_{1}\partial_{221} \varphi \rt).
	\end{align*}
	This expression agrees with \eqref{eq:Avarphi} because $\Delta\varphi=0$, so \eqref{eqx25} proves Lemma~\ref{l:compdivPhismooth}.
\end{proof}

\bibliographystyle{alpha}
\bibliography{aviles_giga}

\newcommand{\etalchar}[1]{$^{#1}$}
\begin{thebibliography}{DMKO01}

\bibitem[ACF{\etalchar{+}}19]{ACFJK}
K.~Astala, A.~Clop, D.~Faraco, J.~J{\"a}{\"a}skel{\"a}inen, and A.~Koski.
\newblock {Improved H{\"o}lder regularity for strongly elliptic PDEs}.
\newblock arXiv:1906.10906, 2019.

\bibitem[ADLM99]{ADM}
L.~Ambrosio, C.~De~Lellis, and C.~Mantegazza.
\newblock Line energies for gradient vector fields in the plane.
\newblock {\em Calc. Var. Partial Differential Equations}, 9(4):327--255, 1999.

\bibitem[AG87]{avgig0}
P.~Aviles and Y.~Giga.
\newblock A mathematical problem related to the physical theory of liquid
  crystal configurations.
\newblock In {\em Miniconference on geometry and partial differential
  equations, 2 ({C}anberra, 1986)}, volume~12 of {\em Proc. Centre Math. Anal.
  Austral. Nat. Univ.}, pages 1--16. Austral. Nat. Univ., Canberra, 1987.

\bibitem[AG96]{avgig1}
P.~Aviles and Y.~Giga.
\newblock The distance function and defect energy.
\newblock {\em Proc. Roy. Soc. Edinburgh Sect. A}, 126(5):923--938, 1996.

\bibitem[AIM09]{ast2}
K.~Astala, T.~Iwaniec, and G.~Martin.
\newblock {\em Elliptic Partial Differential Equations and Quasiconformal
  Mappings in the Plane}, volume~48 of {\em Princeton Mathematical Series}.
\newblock Princeton University Press, 2009.

\bibitem[AIS01]{ast}
K.~Astala, T.~Iwaniec, and E.~Saksman.
\newblock Beltrami operators in the plane.
\newblock {\em Duke Math.J.}, 107, 2001.

\bibitem[AKLR02]{AKLR02}
L.~Ambrosio, B.~Kirchheim, M.~Lecumberry, and T.~Rivi\`ere.
\newblock On the rectifiability of defect measures arising in a micromagnetics
  model.
\newblock In {\em Nonlinear problems in mathematical physics and related
  topics, {II}}, volume~2 of {\em Int. Math. Ser. (N. Y.)}, pages 29--60.
  Kluwer/Plenum, New York, 2002.

\bibitem[BBMN10]{BBMN10}
G.~Bellettini, L.~Bertini, M.~Mariani, and M.~Novaga.
\newblock {$\Gamma$}-entropy cost for scalar conservation laws.
\newblock {\em Arch. Ration. Mech. Anal.}, 195(1):261--309, 2010.

\bibitem[BI76]{boj2}
B.~Bojarski and T.~Iwaniec.
\newblock Quasiconformal mappings and non-linear elliptic equations in two
  variables. i, ii.
\newblock {\em Bull. Acad. Polon. Sci. Sér. Sci. Math. Astronom. Phys.}, 22,
  1976.

\bibitem[BI82]{boj}
B.~Bojarski and T.~Iwaniec.
\newblock Another approach to liouville theorem.
\newblock {\em Math. Nachr.}, 107:253--262, 1982.

\bibitem[Boj74]{boj1}
B.~Bojarski.
\newblock {\em Quasiconformal mappings and general structural properties of
  systems of non linear equations elliptic in the sense of Lavrentev}, volume
  XVIII of {\em Symposia Mathematica (Convegno sulle Transformazioni
  Quasiconformi e Questioni Connesse}.
\newblock INDAM, Rome, 1974.

\bibitem[Bre11]{brezis}
H.~Brezis.
\newblock {\em {Functional analysis, Sobolev spaces and partial differential
  equations.}}
\newblock New York, NY: Springer, 2011.

\bibitem[CDL07]{contidel}
S.~Conti and C.~De~Lellis.
\newblock {Sharp upper bounds for a variational problem with singular
  perturbation.}
\newblock {\em {Math. Ann.}}, 338(1):119--146, 2007.

\bibitem[CET94]{titi}
P.~Constantin, W.~E, and E.~S. Titi.
\newblock Onsager's conjecture on the energy conservation for solutions of
  {E}uler's equation.
\newblock {\em Comm. Math. Phys.}, 165(1):207--209, 1994.

\bibitem[DGF75]{degio}
E.~De~Giorgi and T.~Franzoni.
\newblock Su un tipo di convergenza variazionale.
\newblock {\em Atti Accad. Naz. Lincei Rend. Cl. Sci. Fis. Mat. Natur. 5},
  58:842--850, 1975.

\bibitem[DiP85]{dp}
R.~J. DiPerna.
\newblock Compensated compactness and general systems of conservation laws.
\newblock {\em Trans. Amer. Math. Soc}, 292(2):383--420, 1985.

\bibitem[DL89]{dipernalions}
R.~J. DiPerna and P.-L. Lions.
\newblock Ordinary differential equations, transport theory and {S}obolev
  spaces.
\newblock {\em Invent. Math.}, 98(3):511--547, 1989.

\bibitem[DLI15]{DeI}
C.~De~Lellis and R.~Ignat.
\newblock A regularizing property of the {$2D$}-eikonal equation.
\newblock {\em Comm. Partial Differential Equations}, 40(8):1543--1557, 2015.

\bibitem[DLO03]{ottodel1}
C.~De~Lellis and F.~Otto.
\newblock Structure of entropy solutions to the eikonal equation.
\newblock {\em J. Eur. Math. Soc. (JEMS)}, 5(2):107--145, 2003.

\bibitem[DMKO01]{mul2}
A.~DeSimone, S.~M{\"u}ller, R.~V. Kohn, and F.~Otto.
\newblock A compactness result in the gradient theory of phase transitions.
\newblock {\em Proc. Roy. Soc. Edinburgh Sect. A}, 131(4):833--844, 2001.

\bibitem[FJM02]{MFJ}
G.~Friesecke, R.~D. James, and S.~M\"uller.
\newblock A theorem on geometric rigidity and the derivation of nonlinear plate
  theory from three-dimensional elasticity.
\newblock {\em Comm. Pure Appl. Math.}, 55(11), 2002.

\bibitem[FK12]{FK}
D.~Faraco and J.~Kristensen.
\newblock Compactness versus regularity in the calculus of variations.
\newblock {\em Discrete Contin. Dyn. Syst. Ser. B}, 17(2):473--485, 2012.

\bibitem[FZ05]{farazh}
D.~Faraco and X.~Zhong.
\newblock Geometric rigidity of conformal matrices.
\newblock {\em Ann. Sc. Norm. Super. Pisa Cl. Sci.}, 5(4):557--585, 2005.

\bibitem[Geh62]{gehring}
F.~W. Gehring.
\newblock Rings and quasiconformal mappings in space.
\newblock {\em Trans. Amer. Math. Soc}, 103(3):353--393, 1962.

\bibitem[GL20]{GL}
F.~Ghiraldin and X.~Lamy.
\newblock Optimal besov differentiability for entropy solutions of the eikonal
  equation.
\newblock {\em Comm. Pure Appl. Math.}, 2020.
\newblock to appear.

\bibitem[Gra08]{grafakosclassical}
L.~Grafakos.
\newblock {\em Classical {F}ourier analysis}, volume 249 of {\em Graduate Texts
  in Mathematics}.
\newblock Springer, New York, second edition, 2008.

\bibitem[IM93]{iw1}
T.~Iwaniec and G.~Martin.
\newblock Quasiregular mappings in even dimension.
\newblock {\em Acta Math.}, 170(1), 1993.

\bibitem[IM12]{ignatmerlet12}
R.~Ignat and B.~Merlet.
\newblock Entropy method for line-energies.
\newblock {\em Calc. Var. Partial Differential Equations}, 44(3-4):375--418,
  2012.

\bibitem[Iwa76]{iwan}
T.~Iwaniec.
\newblock Quasiconformal mapping problem for general nonlinear systems of
  partial differential equations.
\newblock {\em Symposia Mathematica}, XVIII, 1976.

\bibitem[JK00]{kohn}
W.~Jin and R.~V. Kohn.
\newblock Singular perturbation and the energy of folds.
\newblock {\em J. Nonlinear Sci.}, 10(3):355--390, 2000.

\bibitem[JOP02]{otto}
P.-E. Jabin, F.~Otto, and B.~Perthame.
\newblock Line-energy {G}inzburg-{L}andau models: zero-energy states.
\newblock {\em Ann. Sc. Norm. Super. Pisa Cl. Sci. (5)}, 1(1):187--202, 2002.

\bibitem[Lio50]{liou}
J.~Liouville.
\newblock Theoreme sur l'equation $dx^2+dy^2+dz^2=\lm (d\alpha^2+d\beta^2+d
  \gamma^2)$.
\newblock {\em J. Math. Pures App.}, 1(15), 1850.

\bibitem[LO18]{xavotto}
X.~Lamy and F.~Otto.
\newblock On the regularity of weak solutions to burgers' equation with finite
  entropy production.
\newblock {\em Calc. Var. Partial Differential Equations}, 57(4), 2018.

\bibitem[Lor14a]{lor10}
A.~Lorent.
\newblock Differential inclusions, non-absolutely convergent integrals and the
  first theorem of complex analysis.
\newblock {\em Q. J. Math.}, 65(4), 2014.

\bibitem[Lor14b]{lor25}
A.~Lorent.
\newblock A quantitative characterisation of functions with low aviles giga
  energy on convex domains.
\newblock {\em Ann. Sc. Norm. Super. Pisa Cl. Sci.}, 13(5), 2014.

\bibitem[LP18]{LP}
A.~Lorent and G.~Peng.
\newblock Regularity of the eikonal equation with two vanishing entropies.
\newblock {\em Ann. Inst. H. Poincar\'{e} Anal. Non Lin\'{e}aire},
  35(2):481--516, 2018.

\bibitem[Mar]{marconi19}
E.~Marconi.
\newblock On the structure of weak solutions to scalar conservation laws with
  finite entropy production.
\newblock arXiv:1909.07257.

\bibitem[MM77]{modmor}
L.~Modica and S.~Mortola.
\newblock Un esempio di {$\Gamma ^{-}$}-convergenza.
\newblock {\em Boll. Un. Mat. Ital. B (5)}, 14(1):285--299, 1977.

\bibitem[MRv05]{mulsv3}
S.~M\"uller, M.~Rieger, and V.~\v{S}ver\'ak.
\newblock Parabolic systems with nowhere smooth solutions.
\newblock {\em Arch. Ration. Mech. Anal.}, 177(1):1--20, 2005.

\bibitem[Mv03]{mulsv2}
S.~M\"uller and V.~\v{S}ver\'ak.
\newblock Convex integration for lipschitz mappings and counterexamples to
  regularity.
\newblock {\em Ann. of Math.}, 9(3):715--742, 2003.

\bibitem[MvY99]{msy}
S.~M\"uller, V.~\v{S}ver\'ak, and B.~Yan.
\newblock Sharp stability results for almost conformal maps in even dimension.
\newblock {\em J. Geom.Anal}, 9(4), 1999.

\bibitem[Pol07]{ark}
A.~Poliakovsky.
\newblock Upper bounds for singular perturbation problems involving gradient
  fields.
\newblock {\em J. Eur. Math. Soc. (JEMS)}, 9(1):1--43, 2007.

\bibitem[Res67]{resh}
Yu.~G Reshetnyak.
\newblock Liouville’s conformal mapping theorem under minimal regularity
  hypotheses. (russian).
\newblock {\em Sibirsk. Mat.\v{Z}}, 8:835--840, 1967.

\bibitem[RS01]{ser1}
T.~Rivi\`ere and S.~Serfaty.
\newblock Limiting domain wall energy for a problem related to micromagnetics.
\newblock {\em Comm. Pure Appl. Math.}, 54(3):294--338, 2001.

\bibitem[RS03]{ser3}
T.~Rivi\`ere and S.~Serfaty.
\newblock Compactness, kinetic formulation, and entropies for a problem related
  to micromagnetics.
\newblock {\em Comm. Partial Differential Equations}, 28(1-2):249--269, 2003.

\bibitem[Sze04]{sz}
L.~Szekelyhidi, Jr.
\newblock The regularity of critical points of polyconvex functionals.
\newblock {\em Arch. Ration. Mech. Anal.}, 172(1):133--152, 2004.

\bibitem[Tri06]{triebel06}
H.~Triebel.
\newblock {\em Theory of function spaces. {III}}, volume 100 of {\em Monographs
  in Mathematics}.
\newblock Birkh\"auser Verlag, Basel, 2006.

\bibitem[\v{S}93]{sverak}
V.~\v{S}ver\'ak.
\newblock On {T}artar's conjecture.
\newblock {\em Ann. Inst. H. Poincar\'e Anal. Non Lin\'eaire}, 10(4):405--412,
  1993.

\bibitem[Zyg02]{zygmund}
A.~Zygmund.
\newblock {\em Trigonometric series. {V}ol. {I}, {II}}.
\newblock Cambridge Mathematical Library. Cambridge University Press,
  Cambridge, third edition, 2002.
\newblock With a foreword by Robert A. Fefferman.

\end{thebibliography}

\end{document}